\documentclass[a4paper,leqno,10pt]{amsart}
\usepackage{amssymb}
\usepackage{amsmath}
\usepackage{amsthm}
\usepackage{graphicx}
\usepackage{texdraw}
\usepackage{epsfig}
\usepackage{amsfonts}
\usepackage{calligra}
\usepackage{hyperref}
\usepackage{mathtools}
\usepackage{soul}
\usepackage{tikz}
\usepackage{bm}
\usepackage{caption}
\usepackage{hyperref}

\newcommand{\R}{\mathbb R}
\newcommand{\N}{\mathbb N}

\definecolor{ao(english)}{rgb}{0.0, 0.5, 0.0}

 \allowdisplaybreaks \numberwithin{equation}{section}
\begingroup
\theoremstyle{plain}
\newtheorem{theorem}{Theorem}[section]
\newtheorem{proposition}[theorem]{Proposition}
\newtheorem{lemma}[theorem]{Lemma}
\newtheorem{corollary}[theorem]{Corollary}

\theoremstyle{definition}
\newtheorem{definition}[theorem]{Definition}
\newtheorem{remark}[theorem]{Remark}

\endgroup

\def \div {\mathop {\rm div}\nolimits}
\def \dive {\mathop {\rm div}\nolimits}

\def \dist {\mathop {\rm dist}\nolimits}
\def \det {\mathop {\rm det}\nolimits}
\def \spt {\mathop {\rm spt}\nolimits}
\def \supp {\mathop {\rm spt}\nolimits}
\def \de {\mathrm{d}}
\def \di {\mathrm{d}}
\def \e {\mathrm{E}}
\def \re {\mathbb R}
\def \RR {\mathcal{R}}

\def \O {\Omega }
\def \Om {\Omega }
\def \C {\mathbb C}

\def \HH {\mathcal H}
\def \E {\mathrm{E}}
\def \F {\mathcal F}
\def \En {\mathcal E}

\def \B {\mathrm{B}}

\def \G {\mathcal G}
\def \M {\mathbb{M}}

\def \II {\mathcal{I}}

\def \ee {\strain}
\def \strain {\boldsymbol{\varepsilon}}

\def \be {\begin{equation}}
\def \ee {\end{equation}}

\def \sbq {\subseteq}
\def \spq {\supseteq}
\def \setdiff { {\setminus} }
\def \wto {\rightharpoonup}
\def \diam {\mathop{\mathrm{diam}}}

\def \Om {\Omega}
\def \Ga {\Gamma}
\def \ga {\gamma}
\def \Haus{\mathcal{H}^1}
\def \K {\mathcal{K}}
\def \eps {\varepsilon}
\def \Lip {\mathrm{Lip}}

\def \adcr {\RR_\eta}

\begin{document}
\author{Stefano Almi}\author{Giuliano Lazzaroni}\author{Ilaria Lucardesi}
\address[Stefano Almi]{Universit\"at Wien, Oskar-Morgenstern-Platz 1, 1090 Vienna, Austria}
\email{stefano.almi@univie.ac.at}
\address[Giuliano Lazzaroni]{Dipartimento di Matematica e Informatica ``Ulisse Dini'',
Universit\`a degli Studi di Firenze, Viale Morgagni 67/a, 50134 Firenze, Italy}
\email{giuliano.lazzaroni@unifi.it}
\address[Ilaria Lucardesi]{Institut \'Elie Cartan de Lorraine, BP 70239, 54506 Vandoeuvre-l\`es-Nancy, France}
\email{ilaria.lucardesi@univ-lorraine.fr}
\thanks{\today}
\title{Crack growth by vanishing viscosity in planar elasticity}
\begin{abstract}
We show the existence of quasistatic evolutions in a fracture model
for brittle materials by a vanishing viscosity approach,
in the setting of planar linearized elasticity.
The crack is not prescribed a priori and is selected in a class of (unions of) regular curves.
To prove the result, it is crucial to analyze the properties of the energy release rate.
\end{abstract}
\maketitle
{\small
\keywords{\noindent {\bf Keywords:}
Free-dis\-con\-ti\-nu\-i\-ty problems;
brittle fracture; crack propagation; 
vanishing viscosity; local minimizers;
energy derivative; 
Griffith's criterion; energy release rate; stress intensity factor.
}
\par
\subjclass{\noindent {\bf 2010 MSC:}
35R35, 
35Q74, 
74R10, 
74G70, 
49J45. 
}
}
\section*{Introduction}
In many applications of engineering, it is crucial to predict the propagation of fracture in structures and to understand whether cracks are stable. When the external loading is very slow if compared with the time scale of internal oscillations (such as, e.g., in a building in standard conditions), it is possible to ignore inertia and to assume that the system is always at equilibrium: the resulting model si called \emph{quasistatic}. Quasistatic (or rate-independent) processes have been extensively analyzed in the mathematical literature both in the context of fracture and of other models (see \cite{MR3380972} and references therein).
\par
The first difficulties in modeling fracture are related to  identifying equilibrium configurations. In fact, in order to state that a configuration is stable, one would have to use a derivative of the mechanical energy with respect to the crack set, which is not well defined. Thus one may prefer a derivative-free formulation where equilibria are restricted to \emph{global} minimizers (of the sum of the mechanical energy and of the dissipated energy due to crack growth), in the context of energetic solutions to rate-independent systems, see e.g.\ \cite{MR1633984,MR1897378,MR1978582,MR1988896,MR2186036,MR2836349,MR3739927}.
\par
A second approach allows one to take into account of more equilibria by restricting the set of the admissible cracks. In fact, the problem is to select a class of regular curves and to prove the existence of a derivative of the mechanical energy with respect to the elongation of a crack in that class. The opposite of this derivative is called \emph{energy release rate} and represents the gain in stored elastic energy due to an infinitesimal crack growth. Griffith's criterion \cite{Grif163} allows crack growth only when the energy release rate reaches the \emph{toughness} of the material (i.e., the energy spent to produce an infinitesimal crack).
\par
In this context, some existence results for crack evolution were given in the case of \emph{antiplane} linear elasticity, where the deformation is represented by a scalar function (that is the vertical displacement, depending on the two horizontal components, while the horizontal displacement is zero). The results of \cite{MR2446401} and \cite{MR2472402} deal with the case of a \emph{prescribed crack}, i.e., before the evolution starts one already knows the set which is going to crack. An algorithm for predicting a stable crack path (chosen from a class of regular curves) was proposed in \cite{MR2851705} and extended in \cite{MR3591293} to a class of curves with branches and kinks. A nonlinear model with vector displacements (still in dimension two) was studied in \cite{MR2658341} for a prescribed crack path.
\par
In this paper we prove an existence result for crack evolution based on Griffith's criterion, in the context of \emph{planar} linear elasticity, in dimension two. In this case the displacement is a vector (with two components). In our model, the path followed by the crack is not a priori known. In fact, the crack is assumed to be the union of a fixed number of $C^{1,1}$ curves and is selected among a class of (unions of) curves with bounded curvature, with no self-intersections, and with at most one point meeting the boundary of the domain (in the reference configuration). Some geometric constraints guarantee that this class is compact with respect to the Hausdorff convergence of sets. The same class of admissible cracks was considered in \cite{MR2851705}.
\par
In order to write the flow rule for crack propagation, we need the expression of the energy release rate. The first step is to prove that, when the crack is a prescribed curve, then the mechanical energy (i.e., the sum of the stored elastic energy and of the work of external volume and surface forces) is differentiable with respect to the arc length of the curve, and its derivative can be written as a surface integral depending on the deformation gradient (Proposition \ref{prop-err}). Since we want a model predicting the crack path (not prescribed a priori), we need to prove that the energy release rate is independent of the extension of the crack (in the class of $C^{1,1}$ curves). This is done in Theorem \ref{t.ERR}. Moreover, the energy release rate is continuous with respect to the Hausdorff convergence of cracks (see Remark \ref{r.ERR2}).
When there are more curves, there is an energy release rate for each crack tip.
\par
Proving such properties of the energy release rate(s) is fundamental to study quasistatic crack evolution and is the major technical difficulty of this work. In fact, the strategy of the proof differs from the method used in the corresponding results in the antiplane case, cf.\ \cite{MR2802893,MR2851705}. 
In planar elasticity, assuming that the crack is $C^{\infty}$, that there are no external forces, and that the elasticity tensor is constant,
it was proven in~\cite{MR3844481} that the energy release rate can be expressed in terms of two stress intensity factors, which characterize the singularities of the elastic equilibrium;
since the stress intensity factors only depend on the current crack, it turns out that the energy release rate is independent of the crack's extension.
In this paper 
we need a corresponding property for $C^{1,1}$ cracks (in the class where we have compactness with respect to Hausdorff convergence) and for energies with external forces and nonconstant elasticity tensor. The same strategy does not apply to the nonsmooth case, in particular we do not prove the existence of the stress intensity factors; nonetheless, we prove that the energy release rate is stable under Hausdorff convergence in the class of $C^{1,1}$ cracks, so we can
employ the results of \cite{MR3844481} 
via some approximation arguments 
(see Section \ref{s.ERR}).
\par
We point out that an energy release rate associated with a crack tip does exist also under much weaker regularity conditions on the crack set. 
For instance, the results of \cite{MR3366673} apply to cracks that are merely closed and connected.
However, in this setting energy release rates can be characterized just up to subsequences through a blow-up limit, thus uniqueness is not guaranteed and, ultimately, the independence on extensions may not hold.
On the other hand, the results of \cite{MR2665275} do not have this limitation, but the initial crack needs to be straight, which makes it impossible to use such characterization in the context of an evolution problem.
(We also refer to \cite{MR3070557} for related results in antiplane elasticity.) For these reasons in this paper we resort to the class of (unions of) $C^{1,1}$ cracks where, as mentioned, better properties can be proven.
\par
This allows us to employ the well known vanishing viscosity method for finding \emph{balanced viscosity solutions} to rate-independent systems, see \cite{MR2446401,MR2851705,MR3591293} for fracture in antiplane elasticity and \cite{MR3380972} for further references. We fix a time discretization and solve some incremental problems where we minimize the sum of the mechanical energy and of the dissipated energy.
Notice that in the present work the dissipated energy density is nonconstant and depends on the position of the crack tip in the reference configuration.
In the minimum problems, the total energy is perturbed with a term penalizing brutal propagations between energy wells, multiplied by a parameter $\eps$.
Passing to the continuous time, we obtain a viscous version of Griffith's criterion, with a regularizing term multiplied by $\eps$; a second passage to the limit as $\eps\to0$ leads to rate-independent solutions.
It is also possible to characterize the time discontinuities of the resulting evolution using the reparametrization technique first proposed in \cite{MR2211809} and then refined in \cite{MR2525194,MR2887927,MR3531671,MR3264231}.
\par
The main result of this paper, extending the results of \cite{MR2851705} to planar elasticity, 
is the existence of a quasistatic evolution (more precisely, a balanced viscosity evolution) fulfilling Griffith's criterion: the length of each component of the crack is a nondecreasing function of time;
at all continuity points of these functions, the energy release rate at each tip is less than or equal to the material's toughness at that tip (which is a stability condition);
the length is increasing only if the energy release rate reaches the toughness.
Moreover, time discontinuities (corresponding to brutal propagation) can be interpolated by a transition, characterized by a viscous flow rule, where the energy release rates are larger than or equal to the toughness (see e.g.\ \cite{MR2745199,MR2915331,MR3021776,MR3332887} for corresponding results in damage and plasticity).
\subsection*{Notation}
Given two vectors $a,b \in \re^d$, their scalar product is denoted by $a \cdot b$. We set $\M^{d}$ the space of $d\times d$ square matrices, and we denote by $\M^{d}_{sym}$ and $\M^d_{skw}$ the subsets of symmetric and skew-symmetric ones, respectively. We set $\mathbf{I}$ the identity matrix in $\M^d$. Given $A$ and $B$ in $\M^d$, we write $A:B$ to denote their Euclidean scalar product, namely $A:B = A_{ij} B_{ij}$. Here and in the rest of the paper we adopt the convention of summation over repeated indices. For every $p\geq1$ we define the $p$-norm in $\mathbb R^d$ as $|x|_p:=\big(\sum_{i=1}^d |x_i|^p\big)^{1/p}$. The 2-norm will be simply denoted by $|\cdot |$. The latter induces the distance $\dist(C,D):= \inf\{|x-y|\,:\, \,x\in C,\,y\in D\}$ between two sets $C$ and $D$. The maximal distance between two points of a set $E$, namely its diameter, is denoted by $\diam(C)$. 

The symbol $\B_\rho(x)$ denotes the open ball of radius $\rho$ in $\R^2$, centred at $x$. The support of a function $f$, namely the closure of $\{f\neq 0\}$, is denoted by  $\spt(f)$. For a tensor field $V\in C^1(\re^d; \M^{d})$, by $\div V$ we mean its divergence with respect to lines, namely $(\div V)_i:=\partial_j V_{ij}$. The symmetric gradient of a vector field $u\in C^1(\re^d;\re^d)$ is denoted by $\E u$, namely $(\E u)_{ij}:=(\partial_i u_j + \partial_j u_i)/2$.

We adopt standard notations for Lebesgue and Sobolev spaces on a bounded open set of $\re^d$. The boundary values of a Sobolev function are always intended in the sense of traces. Boundary integrals on Lipschitz curves are done with respect to the 1-dimensional Hausdorff measure $\mathcal H^1$. Given an interval $I\subset \re$ and a Banach space $X$, $L^p(I; X)$ is the space of $L^p$ functions from $I$ to $X$. Similarly, the sets of continuous and absolutely continuous functions from $I$ to $X$ are denoted by $C^0(I;X)$ and $AC(I;X)$, respectively. Derivatives of functions depending on one variable are denoted by a prime or, when the variable is  time, by a dot.

The identity map in a vector space is denoted by $id$. Given a normed vector space $X$ the norm in $X$ is denoted by $\|\cdot\|_X$. We adopt the same notation also for vector valued functions in $X$. For brevity, the norm in $L^p$ over an open set $\Omega$ of $\mathbb R^d$ is denoted by $\|\cdot \|_{p, \Omega}$ or, when no ambiguity may arise, simply by $\|\cdot \|_p$.
\par


\section{Description of the model and existence results} \label{s:model}
We describe a crack model in \emph{planar elasticity} for a \emph{brittle} body. 
The \emph{body} is represented in its reference configuration by an infinite cylinder $\Om\times\R$, where $\Om\subset\R^2$ is a bounded connected open set, with Lipschitz boundary. 
By assumption, the \emph{displacement} $u$ produced by the external loading is horizontal and depends only on the two horizontal components: the \emph{deformation} is then given by
\[ 
\Om\times\R \ni (x_1,x_2,x_3) \mapsto (x_1+u_1(x_1,x_2),x_2+u_2(x_1,x_2),x_3) \,,
\quad \text{where}\ u=(u_1,u_2)\colon \Om \to \R^2 \,.
\]
\par
\subsection{Admissible cracks}
The set of possible discontinuity points of $u$ (the \emph{crack}) is assumed to lie in a class of admissible regular cracks. We now define such class following \cite{MR2802893,MR2851705}. 
It depends on a parameter $\eta>0$ that is thought as small, but is fixed throughout the paper.
\par
\begin{definition} \label{defin:adcr}
Fixed $\eta>0$, the set $\RR_{\eta}^0$ contains all closed subsets $\Ga\sbq\overline\Om$ such that 
\begin{itemize}
\item[(a)] $\Ga$ is a union of a finite number of arcs of $C^{1,1}$ curves, each of them intersecting $\partial\Om$ in at most one endpoint,
\item[(b)] $\Haus(\Gamma\cap\Om)>0$ and $\Om\setminus\Ga$ is a connected open set, union of a finite number of Lipschitz domains,
\item[(c)]
for every $x\in\Ga$ there exist two open balls $\B_\eta^1,\B_\eta^2 \sbq\R^2$ of radius $\eta$ such that
\be \label{doppia-palla}
\overline{\B}_\eta^1\cap\overline{\B}_\eta^2=\{x\}
\ \text{and}\ 
( \B_\eta^1\cup \B_\eta^2) \cap \Ga = \emptyset \,.
\ee
\end{itemize}
Furthermore, we denote with~$\mathcal{R}^{0,1}_{\eta}$ the class of curves~$\Gamma \in \adcr^{0}$ such that~$\Gamma$ is one arc of curve of class~$C^{1,1}$ intersecting $\partial\Om$ in exactly one endpoint.
\end{definition}
Notice that $\RR_\eta^0\sbq\RR_{\eta'}^0$ if $\eta>\eta'$.
The role of \eqref{doppia-palla} is twofold: on the one hand 
it gives a uniform bound (depending on $1/\eta$) on the curvature of each connected component of any set $\Ga\in\RR_\eta^0$,
on the other hand it ensures that each  of these components is an arc of a simple curve, i.e., a curve with no self-intersections.
Because of $(a)$, each of the arcs has one or two endpoints contained in $\Om$; we say that these points are the \emph{crack tips}.
\par
Since quasistatic models are in general unable to predict crack initiation \cite{MR2385744}, 
i.e., nucleation of a new crack from sound material, 
we assume that there is an \emph{initial crack} $\Ga_0\in\RR_\eta^0$. 
Each connected component of an admissible crack $\Ga$ will be the extension of a connected component of $\Ga_0$, starting from its crack tips.
%
Let $M$ be the number of crack tips of $\Ga_0$; notice that $M$ may be larger than the number of connected components of $\Ga_0$.
We parametrize $\Ga_0$ by introducing $M$ injective functions $\ga^m$  of class $C^{1,1}$, for $m=1,\dots,M$, in the following way:
\begin{itemize}
\item If a connected component $\Ga_0^m$ of $\Ga_0$ intersects $\partial\Om$ in a single endpoint $x_0$, we consider its arc-length parametrization $\ga^m\colon[0,\Haus(\Ga_0^m)]\to\Ga_0^m$ such that $\ga^m(0)\in\partial\Om$ and $\ga^m(\Haus(\Ga_0^m))$ is the crack tip of~$\Ga_0^m$. In particular, a crack in $\mathcal{R}^{0,1}_{\eta}$ has exactly one tip.
\item If a connected component of $\Ga_0$ is contained in $\Om$, we see it as the union of two curves $\Ga_0^m$, $\Ga_0^{m+1}$, intersecting at a single point $\bar x$ (that is not a tip of $\Ga_0$); then we consider two arc-length parametrizations $\ga^m\colon[0,\Haus(\Ga_0^m)]\to\Ga_0^m$, $\ga^{m+1}\colon[0,\Haus(\Ga_0^{m+1})]\to\Ga_0^{m+1}$, such that $\ga^m(0)=\ga^{m+1}(0)=\bar x$ and~$\ga^m(\Haus(\Ga_0^m))$ and $\ga^{m+1}(\Haus(\Ga_0^{m+1}))$ are the two crack tips.
\end{itemize}
We then have $\Ga_0=\bigcup_{m=1}^M \Ga_0^m$. Analogous parametrizations will be used for the extensions of $\Ga_0$.
In the next definition, $M$ is the number fixed above.
\par
\begin{definition} \label{defin:adcr2}
The set $\RR_{\eta}$ contains all subsets $\Ga\in\RR_\eta^0$ 
such that 
\begin{itemize}
\item[(d)] $\Ga$ is the union of $M$ connected subsets $\Ga^1,\dots,\Ga^M$, such that any two of them intersect in up to a point,
\item[(e)] $\Gamma^m\spq \Gamma_{0}^m$ for every $m=1,\dots,M$,
\item[(f)]
for every $m =1,\dots,M$ and for every $x\in\Ga^m\setdiff\Ga_0^m$ 
\[
\B_{2\eta}(x)\cap \Big(\partial\Om \cup \bigcup_{l\neq m} \Ga^l\Big) =\emptyset \,.
\]
\end{itemize}
\end{definition}
Given a set $\Ga=\bigcup_{m=1}^M \Ga^m\in\RR_\eta$, we extend the functions $\ga^m$, $m=1,\dots,M$, defined above, to arc-length parametrizations $\ga^m\colon [0,\Haus(\Ga^m)] \to \Ga^m$; it turns out that they are injective and of class $C^{1,1}$. Properties (a)--(f) ensure that the class $\RR_\eta$ is sequentially compact with respect to Hausdorff convergence (see the next section for details), induced by the following distance.
\par
\begin{definition}\label{def-dh}
Given two compact subsets $\Ga,\Ga'\subset\overline\Om$, their \emph{Hausdorff distance} is given by
\[
d_H(\Ga';\Ga):= \max\left\{ \sup_{x\in \Ga'}\dist(x,\Ga), \ \sup_{x\in \Ga}\dist(x,\Ga') \right\} \,,
\]
with the conventions $d_H(x;\emptyset)=\diam{\Om}$ and $\sup\emptyset=0$.
A sequence $(\Ga_n)_{n\in \mathbb N}$ of compact subsets of $\overline\Om$ \emph{converges to $\Ga$ in the Hausdorff metric} if $d_H(\Ga_n;\Ga)\to0$ as $n\to \infty$.
\end{definition}

\begin{remark}
There are choices of $\Gamma_0$ such that $\RR_\eta$ contains no elements different from $\Ga_0$: we mention a few examples with $\Om=[-1,1]^2$. Let $\Ga_0=[-1,0]\times\{0,\frac14\}$: then $\Ga_0\in\RR_\eta^0$ only if $\eta\le1/8$, thus $\RR_\eta=\emptyset$ if and only if $\eta>1/8$. If instead $\Ga_0=[-1,0]\times\{0\}$, we have $\RR_\eta=\{\Ga_0\}$ if and only if $\eta\ge1/2$.
However, given $\Ga_0$ such that $\RR_\eta$ is trivial, one can find $\eta'<\eta$ such that $\RR_{\eta'}$ contains nontrivial extensions of $\Ga_0$.
Starting from an initial crack with nontrivial extensions, the model described in this paper is reliable as long as our algorithm finds a current configuration $\Ga(t)$ such that there are nontrivial extensions. If, during the evolution, some tip becomes $(2\eta)$-close to $\partial\Om$ or to other connected components of the crack, the results should not be regarded as meaningful.
\end{remark}
\subsection{The mechanical energy and the incremental scheme}
Since the body is brittle, the uncracked part~$\Om\setminus\Ga$ is elastic;
we assume that the displacements are small (so we adopt the setting of linear elasticity) and the crack is traction-free.
We look for evolutions in the time interval $[0,T]$, produced by the time-dependent external loading: 
\begin{itemize}
\item[(H1)] a boundary condition $w\in AC([0,T];H^1(\Om\setdiff\Ga_0;\R^2))$, to be satisfied on a relatively open subset $\partial_D\Om$ of $\partial\Om$, with a finite number of connected components,
\item[(H2)] a volume force $f\in AC([0,T];L^2(\Om\setdiff\Ga;\R^2))$ and a surface force $g\in AC([0,T];L^2(\partial_S\Om;\R^2))$, where~$\partial_S\Om$ is a relatively open subset of $\partial\Om$ such that $\partial_S\Om\Subset\partial\Om\setdiff \overline{\partial_D\Om}$.
\end{itemize}
Without loss of generality, we assume that $\spt(w) \subseteq \{ x \in \overline\Omega \ :\ \dist(x,\partial \Omega) \leq \eta\}$, so $w\equiv0$ around any crack tip.
\par
At each point $x\in\Om$, the stress tensor is $\C(x)\colon \M^{2}_{sym} \to \M^{2}_{sym}$, where
\begin{itemize}
\item[(H3)] $\C(x) A = \lambda(x) \mathrm{tr}(A) \mathbf{I} + 2 \mu(x) A$ for every $A\in\M^{2}_{sym}$, with $\lambda,\mu\in C^{0,1}(\overline\Om)$ such that $\mu(x)>0$ and $\lambda(x) + \mu(x) >0$ for every $x\in\overline\Om$.
\end{itemize}
Notice that the standard conditions $\mu(x)>0$ and $\lambda(x) + \mu(x) >0$ ensure the positive definiteness of $\C(x)$, uniformly in $x$.
\par
Given $t\in[0,T]$ and $\Ga\in\RR_\eta$, the minimum problem
\begin{equation} \label{pb:min}
\min\left\{ \frac 12 \int_{\O\setminus \Gamma} \C \E v : \E v\, \de x - \int_{\O\setminus \Gamma} f(t) \cdot v \, \de x - \int_{\partial_S\Om} g(t) \cdot  v \, \de \Haus 
\colon v\in H^1(\Om\setdiff\Ga;\R^2), \ v=w(t) \ \textrm{on} \ \partial_D\Om \right\} 
\end{equation}
has a unique solution, denoted by
$u(t;\Ga)\colon\Om\setdiff\Ga\to\R^2$, with \emph{elastic energy} 
\[
\En(t;\Ga)\coloneqq \frac 12 \int_{\O\setminus \Gamma_s} \C \E u(t;\Ga) \colon \E u(t;\Ga)\, \de x - \int_{\O\setminus \Gamma} f(t) \cdot  u(t;\Ga) \, \de x - \int_{\partial_S\Om} g(t) \cdot  u(t;\Ga) \, \de \Haus \,.
\]
According to the assumption of brittle behavior, in order to produce a crack the system employs an energy depending (only) on the geometry of the crack itself, in the context of Griffith's theory~\cite{Grif163}.
The \emph{total energy} of the configuration corresponding to a crack $\Ga$ at time $t$ is
\[
\F(t;\Ga):=\En(t;\Ga)+\K(\Ga) \,, \quad \K(\Ga):= \int_{\Ga} \kappa \, \de \Haus \,,
\]
where the surface energy density satisfies
\begin{itemize}
\item[(H4)] $\kappa \in C^{0}(\overline\Om;[\kappa_1,\kappa_2])$,
\end{itemize}
where $0<\kappa_1<\kappa_2$.
\par
Starting from the initial condition $\Ga_0$ fixed above, we define a discrete-time evolution of stable states by solving some incremental minimum problems. 
For every $k\in\N$ we consider a subdivision of the time interval $[0,T]$ in nodes $\{t_{k,i}\}_{0\le i\le k}$ such that
\be \label{defin:time-disc}
0=t_{k,0}<t_{k,1}<\dots<t_{k,k}=T \quad \textrm{and} \quad \lim_{k\to\infty} \max_{1\le i\le k} (t_{k,i}-t_{k,i-1}) =0 \,.
\ee
Fixed $\eps>0$, we define by recursion the sets $\Ga_{\eps,k,i}$, $i=0,\dots,k$, as follows.
We set $\Ga_{\eps,k,0}:=\Ga_0$; for $i\ge1$, $\Ga_{\eps,k,i}$ is a solution to the minimum problem 
\be \label{pb:increm-min-eps}
\min \left\{ \En(t_{k,i};\Ga) + \Haus(\Ga) + \frac{\eps}{2} \, \sum_{m=1}^{M} \frac{\Haus(\Ga^{m}\setdiff\Ga^{m}_{\eps,k,i-1})^2}{t_{k,i}-t_{k,i-1}} 
\colon
\Ga\in\adcr \,, \ \Ga\spq\Ga_{\eps,k,i-1}
\right\} \,,
\ee
where the role of the term multiplied by $\eps$ is to penalize transitions between energy wells.
The existence of solutions to \eqref{pb:increm-min-eps} is proven in Corollary \ref{cor:disc-exist} exploiting the compactness properties of $\adcr$ with respect to the Hausdorff convergence, see Section \ref{s.preliminaries} for details.
\par
We define a piecewise constant interpolation on $[0,T]$ by
\be \label{defin:piecewise}
\Ga_{\eps,k}(0):=\Ga_0 \,, \quad
\Ga_{\eps,k}(t):=\Ga_{\eps,k,i} \quad
\textrm{for} \ t\in(t_{k,i-1},t_{k,i}] \,.
\ee
The \emph{unilateral} constraint $\Ga\spq\Ga_{\eps,k,i-1}$ in \eqref{pb:increm-min-eps} enforces irreversibility of the crack growth, indeed the set function~$t\mapsto \Ga_{\eps,k}(t)$ is nondecreasing with respect to the inclusion.
\subsection{Existence results}
Passing to the limit as $k\to\infty$ and exploiting again the compactness of $\adcr$, we obtain a time-continuous evolution $t\mapsto\Ga_\eps(t)$. In order to understand its properties, we need to define the \emph{energy release rate} associated to a crack.
\par
%
%
%
%
%
For simplicity, let us first consider the case of a \emph{prescribed} curve with only one tip.
Given an increasing family of cracks $\Gamma_{\sigma} \in \mathcal{R}_{\eta}^{0,1}$
parametrized by their arc length~$\sigma\in[0,S]$, we will prove that the map $\sigma\mapsto \En(t;\Gamma_{\sigma})$ is differentiable for every fixed $t$.
Moreover, we will show that the  derivative only depends on the current  configuration~$\Gamma_{s}$, and not on its possible extensions, i.e.,
if $\Gamma_\sigma=\hat\Gamma_\sigma$ for $\sigma\le s$, then 
\[
\frac{\de \En(t;\Gamma_{\sigma})}{\de \sigma} \bigg|_{\sigma=s} =
\frac{\de \En(t;\hat\Gamma_{\sigma})}{\de \sigma} \bigg|_{\sigma=s}.
\]
In particular, we are allowed to write $-\frac{\de \En(t;\Gamma_{\sigma})}{\de \sigma} \bigg|_{\sigma=s}=:\G(t;\Gamma_s)$ with no ambiguity.
The quantity $\G(t;\Gamma_s)$ is the \emph{energy release rate} corresponding to the crack $\Gamma_s$ and
represents the (partial) derivative of the energy~$\En$ with respect to variations of crack in the set of all admissible curves~$\mathcal{R}^{0,1}_{\eta}$ larger than $\Ga_s$.
For the details of these results, we refer to Section \ref{s.ERR} below.
\par
In the case of a curve with several connected components $\Ga\in\adcr$,
for every tip indexed by $m$ we define the $m$-th energy release rate~$\mathcal{G}^{m}(t;\Gamma)$ as above, with respect to variations of the sole component~$\Gamma^{m}$ of~$\Gamma$. 
The energy release rate will be in this case a vector 
$\mathcal{G}(t;\Gamma) \coloneqq (\mathcal{G}^{1}(t;\Gamma), \ldots, \mathcal{G}^{M}(t;\Gamma))$.
\par
%
%
%
%
%
The properties of the evolution $t\mapsto \Gamma_\eps(t)$ are summarized in the next proposition, whose proof is postponed to Section \ref{sec:vve}.
\par
\begin{proposition} \label{prop:viscous}
Fix $\eta>0$, $\Ga_0\in\RR_\eta^0$, and $\eps>0$. Assume \textnormal{(H1)--(H4)}.
Let $\Ga_{\eps,k}$ be as in \eqref{defin:piecewise}.
Then there are a subsequence (not relabeled)
of~$\Ga_{\eps,k}$ and a set function $t\mapsto\Ga_\eps(t)\in\adcr$
such that~$\Ga_{\eps,k}(t)$ converges to~$\Ga_\eps(t)$ in the Hausdorff metric for every $t\in[0,T]$.
\par
Set $\Ga_\eps(t)=\bigcup_{m=1}^M \Ga_\eps^m(t)$, with the conventions of Definition \ref{defin:adcr2}, and 
$l_\eps^m(t):=\Haus(\Ga_\eps^m(t))$.
Then for every $m=1,\dots,M$, and for a.e.\ $t\in[0,T]$
\begin{enumerate}
\item[(G1)$_\eps$] $\dot l_\eps^m(t)\ge 0$;
\item[(G2)$_\eps$] $\kappa(P_\eps^m(t))-\G_\eps^m(t)+\eps \, \dot l_\eps^m(t)\ge0$;
\item[(G3)$_\eps$] $\dot l_\eps^m(t)\,[\kappa(P_\eps^m(t))-\G_\eps^m(t)+\eps \, \dot l_\eps^m(t)]=0$, 
\end{enumerate}
where $\G_\eps^m(t)$ is the energy release rate corresponding to $\Ga_\eps^m(t)$.
\par
 Moreover, along a suitable $\eps$-subsequence,~$\eps \|\dot{l}^{m}_{\eps}\|_{2}^{2}$ is bounded uniformly w.r.t.~$\eps$. 
\end{proposition}
Properties (G1)$_\eps$--(G3)$_\eps$ show that the term multiplied by $\eps$ in \eqref{pb:increm-min-eps} has a regularizing effect, indeed the flow rule for the evolution of $l_\eps:=(l_\eps^1,\dots,l_\eps^M)$ features a time derivative of the unknown. For this reason the corresponding solutions are called \emph{viscous}.
\par
In the passage to the limit as $\eps\to0$, such viscous regularizing term vanishes, so the system follows an evolution of stable states. We thus obtain a \emph{balanced viscosity evolution.}
The next result is proven in Section~\ref{sec:vve}.
\par
\begin{theorem} \label{thm:bv}
Fix $\eta>0$ and $\Ga_0\in\RR_\eta^0$. Assume \textnormal{(H1)--(H4)}.
For every $\eps>0$, let $\Ga_{\eps}$ be the evolution found in Proposition \ref{prop:viscous}.
%
%
Then there are a subsequence (not relabeled)
of $\Ga_{\eps}$ and a set function $t\mapsto\Ga(t)\in\adcr$
such that $\Ga_{\eps}(t)$ converges to $\Ga(t)$ in the Hausdorff metric for every $t\in[0,T]$.
\par
Set $\Ga(t)=\bigcup_{m=1}^M \Ga^m(t)$, with the conventions of Definition \ref{defin:adcr2}, and 
$l^m(t):=\Haus(\Ga^m(t))$.
Then for every $m=1,\dots,M$ 
\begin{enumerate}
\item[(G1)] for a.e.\ $t\in[0,T]$, $\dot l^{m}(t)\ge 0$;
\item[(G2)] for every~$t\in[0,T]$ of continuity for~$l^{m}$, $\kappa(P^m(t))-\G^m(t)\ge0$;
\item[(G3)] for a.e.~$t\in[0,T]$, $\dot l^m(t)\,[\kappa(P^m(t))-\G^m(t)]=0$, 
\end{enumerate}
where $\G^m(t)$ is the energy release rate corresponding to $\Ga^m(t)$.
\end{theorem}
\def \tilt {\tilde{t}}
\def \till {\tilde{l}}
\def \tilga {\widetilde{\Ga}}
\def \tilg {\widetilde{G}}
\def \tilp {\widetilde{P}}
Properties (G1)--(G3) are a formulation of Griffith's criterion for crack growth and show the stability of the evolution $t\mapsto l(t):=(l^1(t),\dots,l^M(t))$  in its continuity points. However, the function $t\mapsto l(t)$ may have discontinuities
and Theorem \ref{thm:bv} does not provide a characterization of jumps in time. The existence result is refined in the following theorem, where we show that
there are a reparametrization of the time interval and a \emph{parametrized evolution}, continuous in time, that interpolates $l$ and follows a viscous flow rule in the intervals corresponding to the discontinuities of $l$.
The next theorem is proven in Section \ref{s:parametric}.
\par
\begin{theorem}[Griffith's criterion] \label{thm:param} 
Fix $\eta>0$ and $\Ga_0\in\RR_\eta^0$. Assume \textnormal{(H1)--(H4)}.
There are absolutely continuous functions $\tilt \colon [0,S]\to [0,T]$ and $\tilga^{m} \colon [0,S] \to \adcr$, $m\in\{1,\ldots, M\}$, such that for a.e.\ $\sigma\in[0,S]$,
setting $\tilga(\sigma)=\bigcup_{m=1}^M \tilga^m(\sigma)$, with the conventions of Definition \ref{defin:adcr2}, and 
$\till^m(\sigma):=\Haus(\tilga^m(\sigma))$,
\begin{itemize}
\item[(pG1)] $\tilt'(\sigma)\ge0$ and $(\till^m)'(\sigma)\ge0$ for every $m=1,\dots,M$;
\item[(pG2)] if $\tilt'(\sigma)>0$, then $\widetilde{\G}^m(\sigma)\le\kappa(\tilp^m(\sigma))$ for every $m=1,\dots,M$;
\item[(pG3)] if $\tilt'(\sigma)>0$ and $(\till^m)'(\sigma)>0$ for some $m\in\{1,\dots,M\}$, then $\widetilde{\G}^m(\sigma)=\kappa(\tilp^m(\sigma))$;
\item[(pG4)] if $\tilt'(\sigma)=0$, then there is $m\in\{1,\dots,M\}$ such that $(\till^m)'(\sigma)>0$;
moreover, for every $m$ with this property, we have $\widetilde{\G}^m(\sigma)\ge\kappa(\tilp^m(\sigma))$,
\end{itemize}
where $\widetilde{\G}^m(\sigma)$ is the energy release rate corresponding to  $\tilga^m(\sigma)$. Moreover, denoting with~$\tilde{u}(\sigma)$ the solution of~\eqref{pb:min} at time~$\tilde{t}(\sigma)$ with a crack~$\widetilde{\Gamma}(\sigma)$, for every $s\in[0,S]$ it holds
\begin{equation}\label{e:vbalance2}
\begin{split}
\F(\tilde{t}(s);\widetilde{\Gamma}(s)) =\ &  \F(0;\Gamma_{0})  + \int_{0}^{s}\int_{\Om}\C\e \tilde u (\sigma) : \e\dot{w}(\tilde{t}(\sigma)) \,\tilde{t}'(\sigma)\,\di x \,\di \sigma \\
& - \sum_{m=1}^M \int_{0}^{s} (\widetilde{\G}^{m} (\sigma) - \kappa(\widetilde P^{m}(\sigma))) (\tilde{l}^{m})'(\sigma)\,\di\sigma \\
& - \int_{0}^{s}\int_{\Om}\dot{f}(\tilde{t}(\sigma))\cdot \tilde u(\sigma) \,\tilde{t}'(\sigma)\,\di x\, \di\sigma - \int_{0}^{s}\int_{\Om} f(\tilde{t}(\sigma))\cdot \dot{w}(\tilde{t}(\sigma)) \,\tilde{t}'(\sigma)\,\di x \,\di \sigma \\
& - \int_{0}^{s} \int_{\partial_{S}\Om} \dot g(\tilde{t}(\sigma))\cdot \tilde u(\sigma) \,\tilde{t}'(\sigma)\,\di \Haus \,\di \sigma - \int_{0}^{s}\int_{\partial_{S}\Om} g(\tilde{t}(\sigma))\cdot \dot{w}(\tilde{t}(\sigma))  \,\tilde{t}'(\sigma) \, \di \Haus \,\di \sigma \,.
\end{split}
\end{equation}

Finally,
\[
\text{if}\ \tilt'(\sigma)>0 \,, \quad \text{then}\ \tilga(\sigma)=\Ga(\tilt(\sigma)) \,,
\]
where $\Ga$ is the balanced viscosity evolution found in Theorem \ref{thm:bv}.
\end{theorem}

\def \adcr {\RR_\eta}
\def \setdiff {\setminus}
\def \diam {\mathrm{diam}}

\section{Preliminary results}\label{s.preliminaries}
In this section we collect some properties of the class of admissible cracks $\RR_\eta$ and of the associated displacements. We recall that, given a crack, the associated displacement is the unique solution to the corresponding minimum problem \eqref{pb:min}.

As already mentioned in the previous section, the elements of $\RR_\eta$ have no self-intersections, and, during the evolution, their crack tips stay uniformly far from the boundary and from the other connected components of the crack set. Moreover, it is easy to show that the curvature and the $\mathcal H^1$ measure of the elements of~$\RR_\eta$ are controlled from above by $\eta^{-1}$ and by some constant $C(\Omega,\Gamma_0,\eta)$, respectively. Finally, as proven in~\cite[Proposition~2.9 and Remark~2.10]{MR2802893}, the class of admissible cracks $\adcr$ is sequentially compact with respect to the Hausdorff convergence introduced in Definition \ref{def-dh}. 
\begin{theorem} \label{thm:comp-adcr}
Every sequence $(\Ga_n)_{n\in \mathbb N}\subset\adcr$ admits (up to a subsequence) a limit $\Ga_\infty\in\adcr$ in the Hausdorff metric.
Moreover, along the subsequence (not relabeled), we have $\Haus(\Ga_n)\to\Haus(\Ga)$ as $n\to \infty$.
\end{theorem}
%
%
%

In what follows we show the continuity of the elastic energy~$\En$ w.r.t.~Hausdorff convergence of the crack set~$\Gamma\in\mathcal{R}_{\eta}$. This will in particular imply the existence of solutions for the incremental minimum problems~\eqref{pb:increm-min-eps}.

 We start with recalling in Proposition~\ref{p.korn1} a Korn inequality for~$\Om\setminus\Gamma$. In Proposition~\ref{p.korn2}, instead, we show that, along sequences of cracks 
$\Gamma_{n}\in\mathcal{R}_{\eta}$ converging in the Hausdorff metric, such an inequality is independent of~$n$. The study is carried out disregarding the time variable, which for brevity is omitted. Accordingly, the elastic energy associated to a fracture $\Gamma$ writes $\mathcal E(\Gamma)$. Furthermore, when explicitly needed, we highlight the dependence on the data by writing $\En( f,g,w, \C ;\Gamma)$ for~$\En(\Gamma)$. 
\begin{proposition}\label{p.korn1}
Let $\Gamma\in\RR_{\eta}$. Then, there exists a positive constant $C=C(\Om,\Gamma)$ such that for every $u\in H^{1}(\Om\setminus\Gamma;\R^{2})$
\[
\|\nabla{u}\|_{2} \leq C (\|u\|_{2} + \|\e u\|_{2})\,.
\]
\end{proposition}

\begin{proof}
Being $\Om\setminus\Gamma$ connected by arcs (see Definition~\ref{defin:adcr2}), it is possible to fix~$\widehat{\Gamma}\supset\Gamma$ such that~$\Omega\setminus \widehat{\Gamma}$ is the union of $N$ disjoint open sets~$\Om_{i}$ with Lipschitz boundaries~$\partial\Om_{i}$ such that $\HH^{1}(\partial_{D}\Om\cap\partial\Om_{i})>0$ for $i\in\{1,\ldots, N\}$, and apply the usual Korn inequality to~$u$ restricted to~$\Om_{i}$.
\end{proof}

\begin{proposition}\label{p.korn2}
Let~$\Gamma_{n},\Gamma_{\infty}\in\RR_{\eta}$ be such that~$\Gamma_{n}$ converges to~$\Gamma_{\infty}$ in the Hausdorff metric as~$n\to\infty$. Then, there exists a positive constant~$C=C(\Om)$ (independent of~$n$) such that for~$n$ sufficiently large
\begin{equation}\label{e.1}
\|\nabla{u}\|_{2}\leq C (\|u\|_{2}+\|\e u\|_{2})\qquad\text{for every $u\in H^{1}(\Om\setminus\Gamma_{n};\R^{2})$}\,.
\end{equation}
Moreover, for~$u\in H^{1}(\Om\setminus\Gamma_{n};\R^{2})$ with~$u=0$ $\Haus$-a.e.~on~$\partial_{D}\Om$ we have
\begin{equation}\label{e.2}
\|\nabla{u}\|_{2}\leq C\|\e u\|_{2}\qquad\text{and}\qquad \|u\|_{2}\leq C\|\e u\|_{2}\,.
\end{equation} 
\end{proposition}

\begin{proof}
 At least for~$n$ sufficiently large, we may assume that there exists an extension~$\widehat{\Gamma}_{n}$ of~$\Gamma_{n}$  such that~$\Om\setminus\widehat{\Gamma}_{n}=\bigcup_{i=1}^{N} \Om_{i}^{n}$, where $\Om_{i}^{n}$ ($i=1,\dots,N$) are open bounded disjoint sets with Lipschitz boundaries and Lipschitz constant~$L$ independent of~$n$. Moreover, we can assume that~$\HH^{1}(\partial_{D}\Om\cap\partial\Om^{n}_{i})>0$ for $i\in\{1,\ldots, N\}$ and every~$n$. The same construction can be repeated for~$n=\infty$ in such a way that $\Om_{i}^{n}$ converges to~$\Om_{i}^{\infty}$ in the Hausdorff metric as~$n\to\infty$.

Let us now fix~$\Om'\Subset \Om_{1}^{\infty}$. For~$n$ large enough (including the case $n=\infty$), we have that~$\Om'\Subset\Om_{1}^{n}$. Hence, applying Proposition~\ref{p.korn1} in~$\Om'$ we deduce that there exists a positive constant~$C'$ independent of~$n$ such that
\begin{equation}\label{e.3}
\|\nabla{u}\|_{2,\Om'} \leq C' (\|u\|_{2,\Om'}+\|\e u\|_{2,\Om'})\qquad\text{for every~$u\in H^{1}(\Om\setminus\Gamma_{n};\R^{2})$}\,.
\end{equation}
Since~$\Om_{1}^{n}$ and~$\Om_{1}^{\infty}$ share the same Lipschitz constant~$L$, applying locally, close to the boundary of~$\Om_{1}^{n}$ (resp.~$\Om_{1}^{\infty}$), the results of~\cite[Theorem~4.2]{MR2350364}, we also obtain that there exists a positive constant~$\widetilde{C}$ such that
\begin{equation}\label{e.4}
\|\nabla{u}\|_{2,\Om_{1}^{n}\setminus \overline{\Om'}}\leq \widetilde{C} (\|u\|_{2,\Om_{1}^{n}\setminus \overline{\Om'}} + \|\e u\|_{2,\Om_{1}^{n}\setminus \overline{\Om'}})\qquad\text{for every~$u\in H^{1}(\Om\setminus\Gamma_{n};\R^{2})$}\,.
\end{equation}
The same inequality can be proven for~$\Om^{n}_{i}$, $i\geq 2$. Therefore, combining~\eqref{e.3} and~\eqref{e.4} we get~\eqref{e.1} for some positive constant~$C$ independent of~$n\in \mathbb{N}\cup\{\infty\}$,~$n$ large enough.

To prove~\eqref{e.2} it is enough to show that
\begin{equation}\label{e.5}
\|u\|_{2} \leq C \|\e u\|_{2}\qquad\text{for every $u\in H^{1}(\Om\setminus\Gamma_{n};\R^{2})$ with~$u=0$ $\Haus$-a.e.~on~$\partial_{D}\Om$, n large enough}\,.
\end{equation}
We proceed with the usual contradiction argument. Assume that~\eqref{e.5} is false. Then, for every~$k\in\mathbb{N}$ there exist~$n_{k}>n_{k-1}$ and~$u_{k}\in H^{1}(\Om\setminus\Gamma_{n_{k}};\R^{2})$ such that $\|u_{k}\|_{2} > k \|\e u_{k}\|_{2}$. Without loss of generality, we may assume that~$\|u_{k}\|_{2}=1$. By~\eqref{e.1} we deduce that~$\|\nabla{u}_{k}\|_{2}$ is bounded. Hence, up to a subsequence,~$\nabla{u}_{k}\rightharpoonup \varphi$ weakly in~$L^{2}(\Om;\M^{2})$ and $u_{k}\to u$ in~$L^{2}_{loc}(\Om\setminus\Gamma_{\infty};\R^{2})$, which implies that $u\in H^{1}_{loc}(\Om\setminus\Gamma_{\infty};\R^{2})$ with~$\nabla{u}=\varphi$. Since $\varphi\in L^{2}(\Om;\M^{2})$, applying~\cite[Proposition~7.1]{MR3614643} we deduce that~$u\in H^{1}(\Om\setminus\Gamma_{\infty};\R^{2})$. Since~$\e u_{k}$ converges to~$0$ in~$L^{2}(\Om;\M^{2})$, we get that $\e u=0$ in~$\Om$. Thus,~$u$ is a rigid movement in~$\Om$, i.e., there exist~$\mathrm{A}\in\M^{2}_{skw}$ and $b\in\R^{2}$ such that $u= \mathrm{A}x+b$ for~$x\in\Om$. Moreover, setting $\Om_{\eta}:=\{x\in\Om:\, \mathrm{dist}(x,\partial\Om)<\eta\}$, by Definition~\ref{defin:adcr2} we have~$(\Gamma_{n}\cap \Om_{\eta}) \setminus \Gamma_{0}=\emptyset$ and $u_{k}\rightharpoonup u$ in~$H^{1}(\Om_{\eta}\setminus\Gamma_{0};\R^{2})$. Therefore, $u=0$ $\Haus$-a.e.~on~$\partial_{D} \Om$, which implies that~$u=0$. We claim that $\|u_{k}\|_{2}\to \|u\|_{2}$. Indeed, $\|u_{k}\|_{2,\Om'}\to \|u\|_{2,\Om'}$ for every~$\Om'\Subset\Om\setminus\Gamma_{\infty}$. By a simple reflection argument applied on both sides of the crack set~$\Gamma_{n}$, we instead obtain that $\|u_{k}\|_{2,\Om\setminus\overline{\Om}'}\to \|u\|_{2,\Om\setminus\overline{\Om}'}$. Thus, $1=\|u_{k}\|_{2}\to \|u\|_{2}=0$, which is a contradiction. This concludes the proof of~\eqref{e.2}.
\end{proof}

We are now ready to prove the continuity of the energy~$\En$ with respect to the crack set. The following lemma is actually stated in a more general setting. Indeed, we show the continuity of the displacement~$u$ solution of~\eqref{pb:min} not only w.r.t.\ the  Hausdorff convergence of sets in~$\mathcal{R}_{\eta}$, but also w.r.t.\ the data of the problem, i.e., the applied forces, the boundary datum, and the elasticity tensor. Such a continuity result will be useful in the next section, where we prove the differentiability of~$\En$ w.r.t.~crack elongations  by using some approximations.

\begin{lemma}\label{l.continuity} 
Let $f_{n},f_{\infty}\in L^{2}(\Om;\R^{2})$, $g_{n}, g_{\infty} \in L^{2}(\partial_{S}\Om;\R^{2})$, $w_{n}, w_{\infty}\in H^{1}(\partial_{S}\Om;\R^{2})$, $\C_{n}, \C_{\infty} \in C^{0,1}(\overline{\Om})$, $\Gamma_{n},\Gamma_{\infty}\in\RR_{\eta}$, and $n\in \mathbb N$ be such that $f_{n} \to f_{\infty}$ strongly in~$L^{2}(\Om;\R^{2})$, $w_{n} \to w_{\infty}$ in~$H^{1}(\Om \setminus\Gamma_{0} ;\R^{2})$, $g_{n} \rightharpoonup g_{\infty}$ weakly in~$L^{2}(\partial_{S}\Om;\R^{2})$, $\C_{n}\to \C_{\infty}$ uniformly in~$\overline\Om$, and $\Gamma_{n}\to\Gamma_{\infty}$ in the Hausdorff metric, as $n\to \infty$.

Then, the energies $\En(f_n,g_n,w_n,\C_n;\Gamma_n)$ converge to $\En(f_\infty, g_\infty,w_\infty,\C_\infty;\Gamma_\infty)$ in the limit as $n\to \infty$. Moreover, the corresponding displacements~$u_n$ and~$u_{\infty}$, solutions to the associated minimum problems~\eqref{pb:min}, satisfy $\nabla u_{n}\to \nabla u_{\infty}$ strongly in $L^{2}(\Om;\M^{2})$.

\end{lemma}

\begin{proof}
The proof is carried out following the steps of \cite[Lemma~3.7]{MR3335535}. The letter $C$ will denote a positive constant, which can possibly change from line to line.

For the sake of clarity, we consider cracks in $\RR_\eta^{0,1}$. The proof can be easily generalized to the whole class~$\RR_\eta$. For brevity, we set $\En_n \coloneqq \En(f_n,g_n,w_n,\C_n;\Gamma_n)$ and $\En_\infty \coloneqq \En(f_\infty, g_\infty,w_\infty,\C_\infty;\Gamma_\infty)$; furthermore, along the proof we denote with~$E_n$ and $E_\infty$ the functionals appearing in the minimization~\eqref{pb:min} with data~$\{f_n,g_n,w_n,\C_n,\Gamma_n\}$ and~$\{f_\infty, g_\infty,w_\infty,\C_\infty,\Gamma_\infty\}$, respectively. Clearly, we have 
\begin{displaymath}
\En_n=E_n(u_n) = \frac 12 \int_{\O} \C \E u_n : \E u_n\, \de x - \int_{\O} f_n \cdot u_n \, \de x - \int_{\partial_S\Om} g_n \cdot  u_n \, \de \Haus\qquad\text{for $n\in\mathbb{N}\cup\{\infty\}$}\,,
\end{displaymath}
where~$\E u_{n}$ are interpreted as functions defined a.e.~in~$\O$.

Let $\gamma_{n}\in C^{1,1}([0,\ell_{n}];\R^{2})$ and $\gamma_{\infty}\in C^{1,1}([0,\ell_{\infty}];\R^{2})$ be the arc-length parametrizations of~$\Gamma_{n}$ and~$\Gamma_{\infty}$, respectively, {where $\ell_{n}$ and $\ell_{\infty}$ denote the $\mathcal H^1$ measures of the crack sets.} By a simple rescaling {of $\gamma_n$}, we construct a $C^{1,1}$ parametrization~$\hat{\gamma}_{n}$ of~$\Gamma_{n}$, defined in $[0,\ell_{\infty}]$. {The new parametrization, by definition of~$\RR_\eta^{0,1}$, belongs to $W^{2,\infty}([0,\ell_{\infty}];\R^{2})$ and its norm is bounded above by a constant independent of $n$.} 
From the Hausdorff convergence of~$\Gamma_{n}$ to~$\Gamma_{\infty}$, we deduce that~$\hat{\gamma}_{n}$ converges to~$\gamma_{\infty}$ weakly* in~$W^{2,\infty}([0,\ell_{\infty}];\R^{2})$ and strongly in~$W^{1,\infty}([0,\ell_{\infty}];\R^{2})$.

Let us fix $\rho>0$ sufficiently small, so that the projection~$\Pi_{\Gamma_{\infty}}$ over~$\Gamma_{\infty}$ is well defined in $\II_{\rho}({\Gamma}_{\infty}):= \{x\in\Om:\, d(x,\Gamma_{\infty})<\rho)\}$. For~$n$ large enough we have  ${\Gamma}_{n}\subseteq\II_{\rho}({\Gamma}_{\infty})$. We want to construct a Lipschitz function $\Lambda_{n}$ such that $\Lambda_{n}(\Gamma_{\infty}) = \Gamma_{n} $ and $\Lambda_{n} (x) = x $ for~$x \in \R^{2} \setminus \II_{\rho} ( \Gamma_{\infty} ) $. For every $x\in\II_{\rho}(\Gamma_{\infty})$ we define $s(x)\in[0,\ell_{\infty}]$ in such a way that $\gamma_{\infty}(s(x))=\Pi_{\Gamma_{\infty}}(x)$. We notice that the map $x\mapsto s(x)$ is locally Lipschitz, while $\Pi_{\Gamma_{\infty}}$ is Lipschitz on~$\II_{\rho}(\Gamma_{\infty})$. Moreover, we set $d_{n}:=\|\hat{\gamma}_{n}- \gamma_{\infty}\|^{1/2}_{W^{1,\infty}}$ and $\lambda_{n}(t):= \big(1-\tfrac{|t|}{d_{n}}\big)_{+}$, where $(\cdot)_{+}$ stands for the positive part.
With this notation at hand, we define
\[
\Lambda_{n}(x):= x+ \lambda_{n}(|x-\Pi_{\Gamma_{\infty}}(x)|) (\hat{\gamma}_{n}(s(x)) - \gamma_{\infty}(s(x))) \qquad\text{for $x\in\R^{2}$}\,.
\]
In particular, $\Lambda_{n}$ is Lipschitz, $\| \Lambda_{n} - id\|_{W^{1,\infty}}\leq C d_{n}\to 0$ as $n\to \infty$, {and, for $n$ large enough,} $\Lambda_{n}(\Gamma_{\infty})=\Gamma_{n}$ and $\Lambda_{n} = id$ out of~$\II_{d_{n}}(\Gamma_{\infty})$. Applying the Hadamard Theorem~\cite[Theorem~6.2.3]{MR1894435}, we deduce that $\Lambda_{n}$ is globally invertible with~$\| \Lambda_{n}^{-1} - id \|_{W^{1,\infty}}\to 0$ as $n\to\infty$.

Given $v\in H^{1}(\Om\setminus\Gamma_{\infty};\R^{2})$ with $v=w_{\infty}$ on~$\partial_D\Om\setminus\Gamma_{0}$, we have that the function $v_{n}:= v\circ \Lambda_{n}^{-1} + w_{n} - w_{\infty}$ belongs to~$H^{1}(\Om\setminus\Gamma_{n};\R^{2})$ and satisfies~$v_{n}= g_{n}$ on~$\partial_D\Om \setminus\Gamma_{0}$. Moreover, $\nabla{v_{n}} \to \nabla{v}$ in $L^{2}(\Om;\M^{2})$, $v_{n}\to v$ in~$L^{2}(\Om;\R^{2})$, and, by the continuity of the trace operator, $v_n \to v$ strongly in $L^2(\partial_S \Omega;\mathbb R^2)$. This asymptotic analysis implies that the sequence $\big(E_n(v_n)\big)_{n\in \mathbb N}$ is bounded and converges to $E_\infty(v)$ as $n\to \infty$.

By the minimality of~$u_{n}$ for $E_n$, we have
\begin{equation}\label{e.6}
E_{n} (u_{n}) \leq E_{n} (v_{n})<C\,.
\end{equation}
It is easy to see that the functionals $E_n$ are equi-coercive in $H^1(\Omega\setminus \Gamma_n;\mathbb R^2)$, so that  inequality \eqref{e.6}, together with Proposition~\ref{p.korn2}, provides a uniform bound on the $L^2$ norm of $u_n$, of its gradient, and of its trace. Therefore, up to a subsequence (not relabeled), we have~$u_{n}\rightharpoonup \varphi$ weakly in~$L^{2}(\Om;\R^{2})$ for some $\varphi \in L^{2}(\Om;\R^{2})$.  Moreover, in a suitably small neighborhood $\mathcal U$ of the boundary, this convergence is stronger, since $(\Omega \setminus \Gamma_n)\cap \mathcal U= (\Omega\setminus \Gamma_0)\cap \mathcal U$ for every $n$. More precisely, we have $u_n\rightharpoonup \varphi$ weakly in $H^1((\Omega\setminus \Gamma_0)\cap \mathcal U;\mathbb R^2)$ and, therefore, $u_n\to \varphi$ strongly in $L^{2}(\partial \Om;\R^{2})$ and $\varphi=w_\infty$ on $\partial_D\O$.
The above convergences imply that
\begin{equation}\label{e:smth}
\lim_{n\to\infty} \int_{\Om} f_{n}\cdot u_{n}\,\di x + \int_{\partial_{S}\Om} g_{n}\cdot u_{n}\,\di\Haus =  \int_{\Om} f_{\infty}\cdot \varphi \,\di x + \int_{\partial_{S}\Om} g_{\infty}\cdot \varphi\,\di\Haus\,.
\end{equation}
Hence, passing to the liminf in \eqref{e.6} we get 
\[
E_{\infty} (\varphi)\leq E_{\infty} (v) \qquad\text{for every $v\in H^{1}(\Om \setminus\Gamma_{\infty}; \R^{2} )$ with $v=w_{\infty}$ $\Haus$-a.e.~on~$\partial_{D}\Om$}\,.
\]
Thus,~$\varphi$ is a minimizer of~$E_{\infty}$ in $H^{1}(\Om\setminus\Gamma_{\infty};\R^{2})$ with boundary condition~$w_{\infty}$. Therefore, by uniqueness of the minimizer, $\varphi= u_{\infty}$. The strong convergence of the gradients follows by considering~\eqref{e.6} for $v=u_{\infty}$. Indeed, we have
\[
E_{\infty} (u_{\infty}) \leq \liminf_{n}\, E_{n} (u_{n}) \leq \limsup_{n} \, E_{n} (u_{n}) \leq \lim_{n} E_{n} (  u_{\infty} \circ \Lambda_{n}^{-1} + w_{n} - w_{\infty}) = E_{\infty} ( u_{\infty}) \,,
\]
which implies, together with~\eqref{e:smth}, that $\En_{n}\to \En_{\infty}$ and $\e u_{n} \to \e u_{\infty}$ in~$L^{2} ( \Om; \M^{2}_{sym} )$. Applying Proposition~\ref{p.korn2} and recalling that~$w_{n} \to w_{\infty}$ in~$H^{1} (\Om \setminus \Gamma_{0} ; \R^{2})$, we also obtain the strong convergence of~$\nabla{u}_{n}$ to~$\nabla{u}_{\infty}$ in~$L^{2} (\Om; \M^{2})$. This concludes the proof of the lemma.
\end{proof}

As a corollary of Lemma~\ref{l.continuity} we deduce the existence of solutions of the incremental minimum problems~\eqref{pb:increm-min-eps}.

\begin{corollary} \label{cor:disc-exist}
Fix $\eps>0$, $k\in\N$, and $i\in\{1,\dots,k\}$. Then the minimum problem~\eqref{pb:increm-min-eps} admits a solution.
\end{corollary}
\begin{proof}
It is sufficient to apply the direct method. Let $(\Gamma_{n})_{n\in\N} \sbq\adcr$ be a minimizing sequence for~\eqref{pb:increm-min-eps}.
By Theorem \ref{thm:comp-adcr}, $\Gamma_{n}$ converges in the Hausdorff metric, up to a subsequence (not relabeled), to some $\Gamma_{\infty}\in \adcr$ such that the constraint $\Ga_\infty\spq\Ga_{\eps,k,i-1}$ is preserved; moreover we have $\Haus (\Gamma_{n})\to \Haus(\Gamma_{\infty})$. Applying Lemma~\ref{l.continuity} with $\C_{n}= \C_{\infty}=\C$, $f_{n}= f_{\infty}= f(t_{k,i})$, $g_{n}=g_{\infty}= g(t_{k,i})$, and $w_{n}= w_{\infty}= w(t_{k,i})$, 
we obtain the convergence of the corresponding energies $\En(t_{k,i};\Ga_n)\to\En(t_{k,i};\Ga_\infty)$.
Hence,~$\Ga_\infty$ is a solution to the minimum problem.
\end{proof}

\section{The energy release rate}\label{s.ERR}

This section is devoted to the definition of the \emph{energy release rate}, i.e., the opposite of the derivative of the  energy~$\En(t;\cdot)$ 
with respect to the crack elongation. The problem is clearly time-independent, therefore we omit the variable $t$, which is kept fixed. As in the previous section, the energy in~\eqref{pb:min} simply writes $\En(\Gamma)$. 

Our aim is to generalize the results obtained in~\cite{MR3844481}, where the energy release rate has been computed only in presence of smooth cracks~$\Gamma$, in the absence of forces, and with a spatially constant elasticity tensor. Here we extend its definition to the case~$\Gamma \in \adcr$, non-zero volume and boundary forces $f\in L^{2}(\Om;\R^{2})$ and $g\in L^{2}(\partial_{S}\Om;\R^{2})$, boundary datum  $w\in H^{1}(\Om\setminus \Gamma_{0};\R^{2})$, and non-constant tensor~$\C\in C^{0,1}(\overline{\Om})$.

As in~\cite{MR3844481}, the fundamental steps are the following:
\begin{itemize}
\item[$(i)$] Given an increasing family of cracks $\Gamma_{\sigma} \in \mathcal{R}_{\eta}^{0,1}$
parametrized by their arc length~$\sigma\in[0,S]$, we prove that the map $\sigma \mapsto \En(\Gamma_{\sigma})$ is differentiable, thus
\[
\frac{\de \En(\Gamma_{\sigma})}{\de \sigma} \bigg|_{\sigma=s} \coloneqq \lim_{\sigma\to s} \frac{\En (\Gamma_{\sigma}) - \En(\Gamma_{s})  }{\sigma-s}\,.
\]

\item[$(ii)$] We show that the above derivative only depends on the current  configuration~$\Gamma_{s}$, and not on its possible extensions, i.e.,
if $\Gamma_\sigma=\hat\Gamma_\sigma$ for $\sigma\le s$, then 
\[
\frac{\de \En(\Gamma_{\sigma})}{\de \sigma} \bigg|_{\sigma=s} =
\frac{\de \En(\hat\Gamma_{\sigma})}{\de \sigma} \bigg|_{\sigma=s} \,.
\]
In particular, we are allowed to write $-\frac{\de \En(\Gamma_{\sigma})}{\de \sigma} \bigg|_{\sigma=s}=:\G(\Gamma_s)$ with no ambiguity.
\end{itemize}

We point out a difference of our strategy with respect to the proof of \cite{MR2802893} for the antiplane case.
In that case, the energy release rate is first characterized via the stress intensity factor assuming that the volume force is null in a neighborhood of the crack tip;
then, one treats general forces by approximation, using the property that the stress intensity factor is continuous with respect to the force.
In this paper, in the planar case we do \emph{not} prove the existence of stress intensity factors for nonsmooth curves.
Hence, when expressing the energy release rate via integral forms, we have to deal carefully with the terms containing the external force.
Once the existence of the energy release rate is guaranteed, we will reduce to the case of forces that are null close to the tip via some approximation arguments, see Lemma \ref{l.approximation} below.

In order to rigorously proceed with~$(i)$, we first restrict our attention to cracks~$\Gamma_{s} \in \mathcal{R}_{\eta}^{0,1}$. We write~$\Gamma_{s}$ as
\begin{equation}\label{crack}
\Gamma_{s}\coloneqq\{\gamma(\sigma):\,0\leq\sigma\leq s\}\,,
\end{equation}
where~$\gamma\in C^{1,1}$ is the arc-length parametrization of~$\Gamma_{s}$. We will discuss in Remark~\ref{r.ERR} how to tackle the general case~$\Gamma\in\adcr$.
%
For brevity, we denote with $u_{s}\in H^{1}(\Om\setminus\Gamma_{s};\R^{2})$ the minimizer of~\eqref{pb:min}. As in the previous section, when explicitly needed, we will highlight the dependence on the data by writing~$\En( f,g,w, \C ;\Gamma_{s})$ for~$\En(\Gamma_{s})$.



In order to make explicit computations, for every $s\in(0,S)$ and~$\delta\in \mathbb R$  with $|\delta|$  small, we need to introduce a $C^{1,1}$ diffeomorphism~$F_{s,\delta}$ that transforms~$\Gamma_{s+\delta}$ in~$\Gamma_{s}$ and maps~$\Om$ into itself. Precisely, for~$r>0$ small enough we may assume that the curve~$\Gamma_{\sigma} \cap \B_{r}(\gamma(s))$,  for $|s-\sigma|$ small , is the graph of a~$C^{1,1}$-function~$\zeta$, with~$\zeta'(\gamma_{1}(s))=0$, where we have denoted with~$\gamma_{1}$ the first component of the arc-length parametrization~$\gamma$. For~$\delta\in \mathbb R$ small  in modulus , we define the function~$F_{s,\delta}\colon \B_{r/2}(\gamma(s)) \to \R^{2}$ by
\[
F_{s,\delta}(x)\coloneqq x+\left(\begin{array}{cc}
(\gamma_{1}(s+\delta)-\gamma_{1}(s))\varphi(x)\\[1mm]
\zeta(x_{1}+(\gamma_{1}(s+\delta)-\gamma_{1}(s))\varphi(x))-\zeta(x_{1})
\end{array}\right)\,,
\]
where~$\varphi\in C^{\infty}_{c}(\B_{r/2}(\gamma(s)))$ is a suitable cut-off function equal to~$1$ close to $\gamma(s)$. Notice that, for $r$ small enough, $\supp(\varphi)\cap\supp(w)=\emptyset$. 
We extend~$F_{s,\delta}$ to the identity in~$\R^{2}\setminus \B_{r/2}(\gamma(s))$. 

By the regularity of~$\zeta$,~$F_{s,\delta}$ is a~$C^{1,1}$ diffeomorphism of~$\R^{2}$ such that $F_{s,\delta}(\Gamma_{s})=\Gamma_{s+\delta}$ and $F_{s,0}=id$. Moreover, the following equalities hold:
\begin{align}
&\rho_{s}(x)\coloneqq \partial_{\delta}(F_{s,\delta}(x))|_{\delta=0}=\gamma'_{1}(s)\varphi(x)\left(\begin{array}{cc} 1\\ \zeta'(x_1)\end{array}\right)\label{someDer}
\\
& \partial_{\delta} (\det\nabla{F_{s,\delta}})|_{\delta=0}=\dive \rho_{s}\,,\qquad \partial_{\delta}(\nabla{F_{s,\delta}})|_{\delta=0}=-\partial_{\delta}(\nabla{F_{s,\delta}})^{-1}|_{\delta=0}=\nabla\rho_{s}\,.\notag
\end{align}
With this notation at hand, we are in a position to prove the differentiability of $s\mapsto \En(\Gamma_{s})$.

\begin{proposition}\label{prop-err}
Let~$\{\Gamma_\sigma\}_{\sigma \geq0} \subseteq \mathcal{R}^{0,1}_{\eta}$ be parametrized as in~\eqref{crack}. Let $f\in L^2(\Omega;\mathbb R^2)$, $g\in L^{2}(\partial_{S}\Om;\R^{2})$, $w\in H^{1}(\Om\setminus\Gamma_{0};\R^{2})$, and $\C\in C^{0,1}(\overline{\Om})$. Then, the map~$\sigma\mapsto \En(\Gamma_{\sigma})$ is differentiable and
\begin{equation}\label{e.err}
\begin{split}
\frac{\de \En(\Gamma_{\sigma})}{\de \sigma} \bigg|_{\sigma=s}= & \ \frac{1}{2}\int_{\Omega  \setminus \Gamma_{s} }(\mathrm{D}\C\,\rho_{s})\nabla{u_{s}} : \nabla{u_{s}}\,\di x -\int_{\Omega\setminus \Gamma_s} \!\!\!\!\!\!\! \C \nabla u_s \nabla \rho_s : \nabla u_s \, \de x \\
&  + \frac 12 \int_{\Omega\setminus \Gamma_s} \!\!\!\!\!\!\! \C\nabla u_s : \nabla u_s \div \rho_s \, \de x + \int_\Omega f \cdot \nabla u_s\,  \rho_s\, \de x\,,
\end{split}
\end{equation}
where $\mathrm{D}\C$ denotes the fourth order tensor
\[
(\mathrm{D}\C\,\rho_{s})_{ijkl}:=\sum_{m=1}^{2}\frac{\partial\C_{ijkl}}{\partial x_{m}}\,\rho_{s,m}\,,\qquad\rho_{s}=(\rho_{s,1},\rho_{s,2})\,.
\]
\end{proposition}

%


\begin{proof}
To prove~\eqref{e.err}, we compute explicitly the limits
\begin{eqnarray}
&\displaystyle \lim_{\sigma\searrow s} \frac{\En (\Gamma_{\sigma}) - \En(\Gamma_{s})  }{\sigma-s} = \lim_{\delta \searrow 0} \frac{\En (\Gamma_{s+\delta}) - \En(\Gamma_{s})  }{\delta}\,,\label{e.limit1}\\
& \displaystyle \lim_{\sigma\nearrow s} \frac{\En (\Gamma_{\sigma}) - \En(\Gamma_{s})  }{\sigma-s} = \lim_{\delta \nearrow 0} \frac{\En (\Gamma_{s+\delta}) - \En(\Gamma_{s})  }{\delta}\,,\label{e.limit2}
\end{eqnarray}
and show that the two limits coincide.

Let us start with~\eqref{e.limit1}. For every $\delta>0$, the function~$u_{s}\circ F_{s,\delta}^{-1}$ belongs to~$H^{1}(\Om\setminus\Gamma_{s+\delta};\R^{2})$ and~$u_{s}\circ F_{s,\delta}^{-1}=w$ on~$\partial_D\Om$. Hence,
\begin{equation}\label{II-6}
\begin{split}
\frac{\En(\Gamma_{s+\delta}) - \En ( \Gamma_{s} )}{\delta}  \leq & \ \frac{1}{2\delta} \bigg( \int_{\Omega \setminus \Gamma_{s+\delta} } \C \nabla{(u_{s}\circ F_{s,\delta}^{-1})} : \nabla{(u_{s}\circ F_{s,\delta}^{-1})} \, \di x - \int_{\Omega   \setminus \Gamma_{s} } \C \e u_{s} : \e u_{s}\,\di x\bigg) \\
& - \frac{1}{\delta}\int_{\Om} f\cdot (u_{s}\circ F_{s,\delta}^{-1} -  u_{s}) \,\di x \,.
\end{split}
\end{equation}
By a change of coordinate in the first integral in~\eqref{II-6} we deduce that
\begin{equation}\label{II-6.1}
\begin{split}
\frac{\En(\Gamma_{s+\delta}) - \En ( \Gamma_{s} )}{\delta}  \leq & \ \frac{1}{2\delta} \bigg( \int_{\Omega  \setminus \Gamma_{s} } \C (F_{s,\delta}) \nabla u_{s} (\nabla F_{s,\delta})^{-1} : \nabla u_{s} (\nabla F_{s,\delta})^{-1} \,\det \nabla F_{s,\delta} \,\di x - \int_{\Omega  \setminus \Gamma_{s} } \C \e u_{s} : \e u_{s}\,\di x\bigg) \\
&- \frac{1}{\delta}\int_{\Om} f\cdot (u_{s}\circ F_{s,\delta}^{-1} -  u_{s}) \,\di x\,.
\end{split}
\end{equation}
By a simple computation, we can rewrite~\eqref{II-6.1} as
\begin{equation}\label{II.6.2}
\begin{split}
 \frac{\En (\Gamma_{s+\delta}) - \En ( \Gamma_{s} )}{\delta}  \leq & \ \frac{1}{2} \int_{\Omega   \setminus \Gamma_{s} } \frac{(\C (F_{s,\delta}) - \C)}{\delta} \nabla u_{s} (\nabla F_{s,\delta})^{-1} : \nabla u_{s} (\nabla F_{s,\delta})^{-1} \,\det \nabla F_{s,\delta} \,\di x  \\
& + \frac{1}{2} \int_{\Om   \setminus \Gamma_{s} } \C \nabla u_{s} \frac{((\nabla F_{s,\delta})^{-1} - \mathbf{I})}{\delta} : \nabla u_{s} (\nabla F_{s,\delta})^{-1} \det \nabla F_{s,\delta}\,\di x\\
& +\frac{1}{2} \int_{\Omega  \setminus \Gamma_{s} } \C  \nabla u_{s}  : \nabla u_{s}  \frac{((\nabla F_{s,\delta})^{-1} - \mathbf{I})}{\delta} \,\det \nabla F_{s,\delta} \,\di x \\
& + \frac{1}{2}\int_{\Omega  \setminus \Gamma_{s} } \C  \nabla u_{s}  : \nabla u_{s} \,\frac{\det \nabla F_{s,\delta} - 1}{\delta} \,\di x
 - \frac{1}{\delta}\int_{\Om} f\cdot (u_{s}\circ F_{s,\delta}^{-1} -  u_{s}) \,\di x \\
 & =: I_{\delta, 1} + I_{\delta, 2} + I_{\delta, 3} + I_{\delta, 4} + I_{\delta, 5}\,.
\end{split}
\end{equation}
Since
\begin{align*}
& \lim_{\delta \to 0} \frac{(\nabla F_{s,\delta})^{-1} - \mathbf{I}}{\delta} =  \partial_\delta (\nabla F_{s,\delta})^{-1}\Big|_{\delta=0} = - \nabla \rho_s\,,
\\
& \lim_{\delta \to 0}\frac{\det \nabla F_{s,\delta} -1 }{\delta} =  \partial_\delta (\det \nabla F_{s,\delta})\Big|_{\delta=0} =  \div \rho_s\,,
\end{align*}
where the limits are uniform in $\delta$,
we obtain
\begin{eqnarray}
&& \displaystyle \lim_{\delta\searrow0} I_{\delta, 1} = \frac12 \int_{\Omega \setminus \Gamma_{s} } (\mathrm{D}\C \rho_{s}) \nabla{u}_s : \nabla u_s\,\de x\,, \label{II.6.3}\\
&& \displaystyle \lim_{\delta\searrow0} I_{\delta, 2} =\lim_{\delta\searrow0} I_{\delta, 3} = - \frac12 \int_{\Omega  \setminus \Gamma_{s} } \C \nabla u_s \nabla\rho_{s}: \nabla u_s  \de x\,, \label{II.6.4}\\
&& \displaystyle \lim_{\delta\searrow0} I_{\delta, 4} = \frac12 \int_{\Omega  \setminus \Gamma_{s} } \C \nabla u_s : \nabla u_s \div \rho_s\, \de x\,.\label{II.6.5}
\end{eqnarray}
Applying e.g.~\cite[Lemma~3.8]{MR3660449} (see also \cite[Lemma~4.1]{MR2394124}), 
we have that $\delta^{-1}(u_s{\,\circ\,}F_{s,\delta}^{-1}-u_s)\to-\nabla{u_s}\,\rho_{s}$ in~$L^{2}(\Omega)$ as $\delta\to0$. Thus,
\begin{equation}\label{II.6.6}
\lim_{\delta\searrow 0} I_{\delta, 5} = \int_{\Om} f\cdot \nabla u_{s}\,\rho_{s}\,\di x\,.
\end{equation}
Combining~\eqref{II.6.2}-\eqref{II.6.6} we get
\begin{equation}\label{II.6.7}
\begin{split}
\limsup_{\delta\searrow 0} \frac{\En (\Gamma_{s+\delta}) - \En ( \Gamma_{s} )}{\delta}  \leq  &\ \frac12 \int_{\Omega  \setminus \Gamma_{s} } (\mathrm{D}\C \rho_{s}) \nabla{u}_s : \nabla u_s\,\de x\\
& -  \int_{\Omega  \setminus \Gamma_{s}  } \C \nabla u_s \nabla\rho_{s}: \nabla u_s  \de x + \frac12 \int_{\Omega  \setminus \Gamma_{s} } \C \nabla u_s : \nabla u_s \div \rho_s\, \de x\\
& + \int_{\Om} f\cdot \nabla u_{s}\,\rho_{s}\,\di x\,.
\end{split}
\end{equation}

In order to obtain the opposite inequality, we consider the function~$u_{s+\delta}\circ F_{s,\delta}\in H^{1}(\Om\setminus\Gamma_{s};\R^{2})$. By the minimality of~$u_{s}$ we have that
\begin{equation}\label{II.6.7}
\begin{split}
\frac{\En(\Gamma_{s+\delta}) - \En(\Gamma_{s})}{\delta}\geq & \ \frac{1}{2\delta} \bigg( \int_{\Omega  \setminus \Gamma_{s + \delta} } \C \e u_{s+\delta} : \e u_{s+\delta}\,\di x -  \int_{\Omega  \setminus \Gamma_{s} } \C \nabla{(u_{s+\delta}\circ F_{s,\delta})} : \nabla{(u_{s+\delta}\circ F_{s,\delta})} \, \di x\bigg) \\
& - \frac{1}{\delta}\int_{\Om} f\cdot (u_{s+\delta} - u_{s+\delta}\circ F_{s,\delta} ) \,\di x \,.
 \end{split}
\end{equation}
For simplicity of notation, we denote with~$U_{s,\delta}\coloneqq u_{s+\delta}\circ F_{s,\delta}$. By a change of variable in the first integral in~\eqref{II.6.7} and we deduce that
\begin{equation}\label{II.6.7}
\begin{split}
& \frac{\En(\Gamma_{s+\delta}) - \En(\Gamma_{s})}{\delta} \\
&\qquad\qquad \geq  \frac{1}{2\delta} \bigg( \int_{\Omega \setminus \Gamma_s} \C (F_{s,\delta}) \nabla U_{s,\delta} (\nabla F_{s,\delta})^{-1} : \nabla U_{s,\delta} (\nabla F_{s,\delta})^{-1} \det \nabla F_{s,\delta}\,\di x -  \int_{\Omega\setminus \Gamma_s} \C \nabla{U_{s,\delta}} : \nabla{U_{s,\delta}} \, \di x\bigg) \\
& \qquad\qquad\qquad - \frac{1}{\delta}\int_{\Om} f\cdot (u_{s+\delta} - U_{s,\delta} ) \,\di x \,.
 \end{split}
\end{equation}
Repeating the computations of~\eqref{II.6.2}-\eqref{II.6.7} and taking into account that $\delta^{-1}(u_{s+\delta}-U_{s,\delta})\rightharpoonup -\nabla{u}\,\rho_{s}$ weakly in~$L^{2}(\Om;\R^{2})$ (see, e.g.,~\cite[Lemma~3.8]{MR3660449}), we infer that
\[
\begin{split}
\liminf_{\delta\searrow 0} \frac{\En(\Gamma_{s+\delta}) - \En(\Gamma_{s})}{\delta}\geq &\ \frac12 \int_{\Omega \setminus \Gamma_{s} } (\mathrm{D}\C \rho_{s}) \nabla{u}_s : \nabla u_s\,\de x\\
& -  \int_{\Omega \setminus \Gamma_{s} } \C \nabla u_s \nabla\rho_{s}: \nabla u_s  \de x + \frac12 \int_{\Omega\setminus \Gamma_{s}  } \C \nabla u_s : \nabla u_s \div \rho_s\, \de x\\
& + \int_{\Om} f\cdot \nabla u_{s}\,\rho_{s}\,\di x\,,
\end{split}
\]
which, together with~\eqref{II.6.7} implies that
\[
\begin{split}
\lim_{\delta\searrow 0} \frac{\En(\Gamma_{s+\delta}) - \En(\Gamma_{s})}{\delta}= &\ \frac12 \int_{\Omega\setminus \Gamma_{s} } (\mathrm{D}\C \rho_{s}) \nabla{u}_s : \nabla u_s\,\de x\\
& -  \int_{\Omega \setminus \Gamma_{s} } \C \nabla u_s \nabla\rho_{s}: \nabla u_s  \de x + \frac12 \int_{\Omega \setminus \Gamma_{s}  } \C \nabla u_s : \nabla u_s \div \rho_s\, \de x\\
& + \int_{\Om} f\cdot \nabla u_{s}\,\rho_{s}\,\di x\,.
\end{split}
\]

Adapting the above argument to the case~$\delta<0$, cf.\ \eqref{e.limit2}, it is also possible to prove that
\begin{displaymath}
\begin{split}
\lim_{\delta\nearrow 0} \frac{\En(\Gamma_{s+\delta}) - \En(\Gamma_{s})}{\delta}= &\ \frac12 \int_{\Omega \setminus \Gamma_{s} } (\mathrm{D}\C \rho_{s}) \nabla{u}_s : \nabla u_s\,\de x\\
& -  \int_{\Omega \setminus \Gamma_{s}  } \C \nabla u_s \nabla\rho_{s}: \nabla u_s  \de x + \frac12 \int_{\Omega \setminus \Gamma_{s}  } \C \nabla u_s : \nabla u_s \div \rho_s\, \de x\\
& + \int_{\Om} f\cdot \nabla u_{s}\,\rho_{s}\,\di x\,.
\end{split}
\end{displaymath}
This concludes the proof of~\eqref{e.err}.
\end{proof}

The following corollary states the continuity of the derivative~\eqref{e.err} w.r.t.~the data $f$,~$g$,~$w$,~$\C$, and~$\Gamma_{s}$.
%

\begin{corollary}\label{c.continuityERR}
Let $f_{n}, f \in L^{2}(\Om;\R^{2})$, $g_{n}, g\in L^{2}(\partial_{S}\Om;\R^{2})$, $w_{n},w\in H^{1}(\Om\setminus\Gamma_{0};\R^{2})$, and $\C_{n}, \C\in C^{0,1}(\overline{\Om})$ be such that $f_{n}\to f$ strongly in $L^{2}(\Om;\R^{2})$, $g_{n} \rightharpoonup g$ weakly in $L^{2}(\partial_{S}\Om;\R^{2})$, $w_{n} \rightharpoonup w$ weakly in $H^{1}(\Om\setminus\Gamma_{0};\R^{2})$, and $\C_{n} \rightharpoonup \C$ weakly* in $W^{1,\infty}(\Om)$. Moreover, let $S>0$, let $\{\Gamma_{s}\}_{s\in[0,S]} \subseteq \mathcal{R}^{0,1}_{\eta}$ be as in~\eqref{crack}, and assume that there exists a sequence $\{\Gamma^{n}_{s}\}_{s\in[0,S]} \subseteq \mathcal{R}^{0,1}_{\eta}$ such that~$\Gamma^{n}_{s}$ converges to~$\Gamma_{s}$ in the Hausdorff metric of sets for every $s\in[0,S]$. Then, for every $s\in(0,S)$ we have
\begin{equation}\label{e.continuityERR}
\lim_{n}\, \frac{\de\En (f_{n}, g_{n}, w_{n},\C_{n};\Gamma^{n}_{\sigma})}{\de \sigma} \bigg|_{\sigma = s} = \frac{\de\En (f, g, w, \C;\Gamma_{\sigma} )}{\de \sigma} \bigg|_{\sigma = s} \,.
\end{equation}
\end{corollary}

\begin{proof}
Let us denote with $u_{n}$ and~$u$ the displacements associated to~$\En(f_{n}, g_{n}, w_{n},\C_{n};\Gamma^{n}_{s})$ and to~$\En(f, g, w, \C;\Gamma_{s})$, respectively. By Lemma~\ref{l.continuity} and by the hypotheses, it follows that~$\nabla u_{n}$ converges to~$\nabla u$ strongly in~$L^{2}(\Om;\mathbb{M}^{2})$. Let us denote by $\rho^{n}_{s}$ the quantity defined as in~\eqref{someDer} and corresponding to $\Gamma^{n}_{s}$. Since $\Gamma_{s}^{n}$ converges to~$\Gamma_{s}$ in the Hausdorff metric of sets for every~$s\in[0,S]$, we have that~$\rho^{n}_{s} \to \rho_{s}$ uniformly in~$\Om$  and weakly* in $W^{1,\infty}(\Om;\R^{2})$ for every $s\in(0,S)$.  Thus~\eqref{e.continuityERR} follows by \eqref{e.err}.
\end{proof}

We notice that the dependence of~$\frac{\di \En (\Gamma_{\sigma})}{\di \sigma} \big|_{\sigma = s}$ on~$\{\Gamma_{\sigma}\}_{\sigma\in[0,S]}$ is encoded by the quantity~$\rho_{s}$ introduced in~\eqref{someDer}. The rest of this section is devoted to step {\it (ii)}, namely at proving that the above derivative only depends on the current fracture~$\Gamma_{s}$, and not on its possible extensions, i.e., on the choice of the family~$\{\Gamma_{\sigma}\}_{\sigma \in[0,S]}$. We start by recalling a result of~\cite{MR3844481} stating that this very same property holds for~$C^{\infty}$ cracks in absence of external forces and with constant elasticity tensor.

\begin{theorem}[{\cite[Theorem 4.1]{MR3844481}}]\label{t.AL18}
Let $f=g=0$ and let~$\C$ be constant in~$\Om$. Let~$\{\Gamma_{\sigma}\}_{\sigma\in[0,S]} \subseteq \mathcal{R}^{0,1}_{\eta}$ be as in~\eqref{crack} and assume that there exists $s \in (0, S)$ such that~$\Gamma_{\sigma}$ is of class~$C^{\infty}$ for every $\sigma \in(0,\bar{s}]$. Then, for every~$s\in(0,\bar{s}]$ there exist two constants~$Q_1(\Gamma_s)$ and~$Q_2(\Gamma_s)$ (independent of~$\Gamma_{\sigma}$ for $\sigma > s$) such that
\begin{equation}\label{e.AL18}
\frac{\de \En (\Gamma_{\sigma})}{\de \sigma} \bigg|_{\sigma = s} = C(\lambda,\mu) (Q_1^2(\Gamma_s) + Q_2^2(\Gamma_s)) \,,
\end{equation}
where~$C(\lambda,\mu)$ is a constant which depends only on the Lam\'e coefficients~$\lambda$ and~$\mu$.
\end{theorem}

\begin{remark}\label{r.AL18}
The constants~$Q_1(\Gamma_s)$ and~$Q_2(\Gamma_s)$ are the so called \emph{stress intensity factors}. Indeed, it has been proven in~\cite[Theorem~2.5]{MR3844481} that, in the condition of~Theorem~\ref{t.AL18},  the displacement~$u_{s}$ can be written as
\begin{equation}\label{e.AL18-2}
u_{s}= u_{R} + Q_1(\Gamma_s) \Phi_{1} + Q_2(\Gamma_s) \Phi_{2}\,,
\end{equation}
for suitable functions $u_{R} \in H^{2}(\Om \setminus \Gamma_{s}; \R^{2} )$ and $\Phi_{1},\Phi_{2}\in H^{1} ( \Om \setminus \Gamma_{s}; \R^{2} ) \setminus H^{2}(\Om\setminus\Gamma_{s} ;\R^{2})$. Moreover, the proof of formula~\eqref{e.AL18} follows from the above decomposition.
\end{remark}

The next proposition is a simple localization of Theorem~\ref{t.AL18}. 

\begin{proposition}\label{p.localized}
Let $\{\Gamma_{s}\}_{s\in [0,S]} \subseteq \mathcal{R}^{0,1}_{\eta}$ be as in~\eqref{crack}. Let~$ s \in(0,S)$, $f\in L^{2}(\Om;\R^{2})$, $g\in L^{2}(\partial_{S}\Om;\R^{2})$, $w\in H^{1}(\Om\setminus\Gamma_{0};\R^{2})$, and $\C\in C^{0,1}(\overline{\Om})$ be such that $\Gamma_{s}$ is~$C^{\infty}$, $f=0$, and $\C$ is constant in a neighborhood of the tip~$\gamma(s)$ of~$\Gamma_{s}$. Then, there exist two constants~$Q_1(\Gamma_{s})$ and~$Q_2(\Gamma_{s})$ (independent of~$\Gamma_{\sigma}$ for $\sigma  >   s$) such that
\begin{equation}\label{e.AL18-3}
\frac{\de \En(\Gamma_{\sigma } ) }{\de \sigma} \bigg|_{\sigma = s} = C(\lambda_{s},\mu_{s}) (Q_1^2(\Gamma_{s}) + Q_2^2(\Gamma_{s})) \,,
\end{equation}
where~$C(\lambda_{s},\mu_{s})$ coincides with the constant appearing in~\eqref{e.AL18} and~$\lambda_{s},\mu_{s}$ denote the Lam\'e coefficients of~$\C$ in~$\gamma(s)$.
\end{proposition}

\begin{proof}
As mentioned in Remark~\ref{r.AL18}, the proof of formula~\eqref{e.AL18-3} follows directly from a splitting of the form~\eqref{e.AL18-2} for the displacement~$u_{s}$ solution of
\begin{displaymath}
\min\, \left \{ \frac12 \int_{\Om\setminus\Gamma_{s}}\!\!\!\!\!\!\! \C\e u: \e u\,\de x -\int_{\Om\setminus\Gamma_{s}}\!\!\!\!\!\! f\cdot u \,\de x - \int_{\partial_{S}\Om}\!\!\!\! g\cdot u\,\de \Haus :\, u\in H^{1}(\Om\setminus\Gamma_{s}; \R^{2} ) , \, u=w \text{ on $\partial_{D}\Om$}\right \} .
\end{displaymath}
close to the tip~$\gamma(s)$ of~$\Gamma_{s}$. Indeed, given~\eqref{e.AL18-2}, we can simply repeat step by step the proof of~\cite[Theorem~4.1]{MR3844481} and get~\eqref{e.AL18-3}. In order to obtain such a decomposition in a neighborhood of~$\gamma(s)$, we note that~$u_{s}$ is also solution of
\begin{displaymath}
\min\,\left\{  \frac12 \int_{\B_{\ell}(\gamma(s)) \setminus \Gamma_{s}} \!\!\!\!\!\!\!\! \!\!\!\!\!\!\!\!\!\!\!\!\!\! \C\e u: \e u\,\de x :\, u\in H^{1}(\B_{\ell}(\gamma(s)) \setminus \Gamma_{s}; \R^{2} ) , \, u=u_{s} \text{ on $\partial \B_{\ell}(\gamma(s)) $ } \right\}
\end{displaymath}
with~$\ell$ chosen in such a way that~$\Gamma_{s}$ is smooth, $f=0$, and~$\C$ is constant in~$\B_{\ell}(\gamma(s))$. This enables us to apply~\cite[Theorem~2.5]{MR3844481} on the domain~$\B_{\ell}(\gamma(s))$ and to deduce the decomposition~\eqref{e.AL18-3} on~$\B_{\ell}(\gamma(s))$.
\end{proof}

We are now in a position to state and prove the main result of this section.

\begin{theorem}\label{t.ERR}
Let~$\{\Gamma_\sigma\}_{\sigma\in[0,S]}, \{\hat{\Gamma}_{\sigma}\}_{\sigma\in[0,S]}  \subseteq \mathcal{R}^{0,1}_{\eta}$ be as in~\eqref{crack}. Let $f\in L^2(\Omega;\mathbb R^2)$, $g\in L^{2}(\partial_{S}\Om;\R^{2})$, $w\in H^{1}(\Om\setminus\Gamma_{0};\R^{2})$, and $\C\in C^{0,1}(\Om)$. Let $s\in(0,S)$ and assume that $\Gamma_{\sigma} = \hat{\Gamma}_{\sigma}$ for $\sigma \leq s$. Then, 
\begin{equation}\label{e:finalERR}
\frac{\de \En (\Gamma_{\sigma})}{\de\sigma} \bigg|_{\sigma = s} = \frac{\de \En (\hat \Gamma_{\sigma})}{\de\sigma} \bigg|_{\sigma = s}.
\end{equation}
\end{theorem}

\begin{remark}
The previous theorem states that the derivative~$\frac{\de \En (\Gamma_{\sigma})}{\de \sigma} \big|_{\sigma = s}$ computed in Proposition~\ref{prop-err} does not depend on the possible extension of~$\Gamma_{s}$ in the class~$ \mathcal{R}^{0,1}_{\eta}$. Hence, it represents the slope of the energy~$\En$ with respect to variations of crack in the set of admissible curves~$\mathcal{R}^{0,1}_{\eta}$. 
\end{remark}

The proof of Theorem~\ref{t.ERR} is a corollary of the following lemma, where we use an approximation argument to reduce ourselves to the case of smooth cracks, constant elasticity tensor, and forces that are null close to the crack tip. In the latter case, the relation between the energy release rate and the stress intensity factors shows \eqref{e:finalERR}, cf.\ \cite{MR3844481}.

\begin{lemma}\label{l.approximation}
Let~$\{\Gamma_{\sigma}\}_{\sigma\in[0,S]}$, $f$, and~$\C$ be as in the statement of Theorem~\ref{t.ERR}, and let~$s\in(0,S)$. Then, there exist~$\delta>0$,~$f_{n}\in L^{2}(\Om;\R^{2})$,~$\C_{n} \in C^{0,1}(\overline{\Om})$, and~$\{\Gamma_{\sigma}^{n}\}_{\sigma\in [0,s + \delta]}  \subseteq \mathcal{R}^{0,1}_{\eta}$ such that~$f_{n} \to f$ strongly in~$L^{2}(\Om;\R^{2})$,~$\C_{n}\rightharpoonup \C$ weakly* in~$W^{1,\infty}(\Om)$,~$\Gamma^{n}_{\sigma}\to \Gamma_{\sigma}$ in the Hausdorff metric of sets for every~$\sigma\in [0, s+\delta]$, and, close to the tip of~$\Gamma_{s}^{n}$,~$\Gamma_{s}^{n}$ is smooth, $f_{n}=0$, and~$\C_{n}$ is constant.

Moreover, if~$\{\hat{\Gamma}_{\sigma}\}_{\sigma\in[0,S]}$ is another family of curves in~$\adcr^{0,1}$ with~$\hat{\Gamma}_{\sigma} = \Gamma_{\sigma}$ for $\sigma\leq s$, then the sequences~$\{\Gamma_{\sigma}^{n}\}_{\sigma\in [0,s + \delta]}$,~$\{\hat{\Gamma}_{\sigma}^{n}\}_{\sigma\in [0,s + \delta]}$,~$\C_{n}$,~$\hat{\C}_{n}$,~$f_{n}$, and~$\hat{f}_{n}$ can be chosen in such a way that~$\hat{\Gamma}_{\sigma}^{n} = \Gamma_{\sigma}^{n}$ for~$\sigma\leq s$,~$\C_{n} = \hat{\C}_{n}$, and~$f_{n}= \hat{f}_{n}$.
\end{lemma}

\begin{proof}
We start with the construction of an approximating sequence for~$\Gamma_{\sigma}$. Let $\delta>0$ be such that~$s+\delta<S$. Let us fix a sequence~$s_{n} \nearrow s$. By definition of the class~$\adcr^{0,1}$, for every~$n$ there exist two open balls~$\B_{\eta,n}^1$ and~$\B_{\eta,n}^2$ of radius~$\eta$ such that $\overline{\B}_{\eta,n}^1\cap\overline{\B}_{\eta,n}^2=\{\gamma(s_{n}) \}$ and $(\B_{\eta,n}^1\cup \B_{\eta,n}^2)\cap\Gamma_{ s}=\emptyset$. Up to a redefinition of~$\delta$, for~$n$ large enough we may assume that the portion of curve $\{\gamma(\sigma):\, s_{n}-\delta \leq \sigma \leq s + \delta\}\subseteq \Gamma_{s + \delta}$ can be represented, in a suitable coordinate system~$(x_{1}, x_{2})$  possibly dependent on $n$, as graph of a function~$\psi_{n}$ of class~$C^{1,1}$ with~$\psi'_{n}(x_{1}^{s_{n}})=0$, where the point $(x_{1}^{s_{n}}, \psi_{n}(x_{1}^{s_{n}}))$ coincides with~$\gamma(s_{n})$. A similar notation is used for~$\gamma( s) = (x_{1}^{s}, \psi_{n}(x_{1}^{s}))$. Without loss of generality, we assume that~$\psi_{n}'(x_{1}^{ s}) \geq 0$. 

The idea of our construction is to extend the curve~$\Gamma_{s_{n}}$ with the arc of circumference of equation
\begin{equation}\label{e:crf}
x_{2} = \psi_{n}(x_{1}^{s_{n}}) + \eta - \sqrt{ \eta^{2} - (x_{1} - x_{1}^{s_{n}})^{2} } \qquad\text{for $x_{1} \in [x_{1}^{s_{n}}, \bar x_{1})$}\,,
\end{equation}
where~$\bar{x}_{1}$ is the smallest~$x_{1}\geq x_{1}^{s_{n}}$ such that $\psi_{n}'(\bar{x}_{1}) = \psi_{n}'(x_{1}^{ s})$. We notice that~\eqref{e:crf} is the equation of the boundary of one of the two open balls~$\B_{\eta, n}^{i}$ and that $\bar{x}_{1} = x_{1}^{s_{n}}$ whenever $\psi_{n}'(x_{1}^{s}) = 0$. We denote with~$\Lambda^{n}$ the extension of~$\Gamma_{s_{n}}$ with the arc~\eqref{e:crf}  and its tip with~$P_{n}$. We also use the symbol~$\Lambda_{\sigma}^{n}$, $\sigma\in[0,\Haus(\Lambda^{n})]$, to indicate the piece of curve contained in~$\Lambda^{n}$ of length~$\sigma$.

A direct computation gives $\Haus(\Lambda^n)=s_n +\eta \arctan(\psi_{n}'(x_{1}^{s}))$, which can also be written as follows:
\begin{equation}\label{Hn}
\Haus(\Lambda^n) = s_n +\eta \int_0^{\psi_{n}'(x_{1}^{s})}\frac{1}{1+x^2}\, \mathrm{d}x.
\end{equation}
On the other hand, exploiting the upper bound $\eta^{-1}$ on the curvature of the crack set, which reads $|\psi_n''|\,[1+(\psi_n')^2]^{-3/2}\leq \eta^{-1}$ in terms of the graph parametrization, we get
\begin{equation}\label{Hsbar}
\Haus(\Gamma_{s})=s_n + \int_{x_1^{s_n}}^{x_1^s}\sqrt{1+(\psi_n')^2(x)}\, \mathrm{d} x\geq s_n +  \eta  \int_{x_1^{s_n}}^{x_1^s} \frac{|\psi_n''|}{1+(\psi_n')^2(x)}\, \mathrm{d} x \geq s_n +  \eta  \int_0^{\psi_{n}'(x_{1}^{s})} \frac{1}{1+x^2}\, \mathrm{d} x.
\end{equation}
Comparing \eqref{Hn} and \eqref{Hsbar}, we conclude that $\Haus(\Lambda^{n}) \leq \Haus(\Gamma_{s})=s$.

If~$\Haus(\Lambda^{n}) < s$, we denote with~$\alpha^{n}_{\ell}$ the segment of length~$\ell$, initial point~$P_{n}$ and parallel to~$\gamma'( s)$, and we define~$\{\Gamma^{n}_{\sigma}\}_{\sigma\in[0,s + \delta]}$ as follows:
\begin{displaymath}
\Gamma^{n}_{\sigma}\coloneqq \left\{ \begin{array}{lll}
\Lambda^{n}_{\sigma} & \text{if $\sigma\in[0, \Haus(\Lambda^{n})]$}\,,\\[1mm]
\Lambda^{n} \cup \alpha^{n}_{\sigma - \Haus(\Lambda^{n})}  & \text{if $\sigma\in[\Haus(\Lambda^{n}), s ]$}\,,\\[1mm]
\Lambda^{n} \cup \alpha^{n}_{ s - \Haus(\Lambda^{n})}  \cup \big( (P_{n} + (s - \Haus(\Lambda^{n})) \gamma'(s)  - \gamma(s)) + \Gamma_{\sigma}\setminus\Gamma_{s }\big) & \text{if $\sigma\in (\Haus(\Lambda^{n}) , s + \delta]$}\,,
\end{array}\right.
\end{displaymath}
where we have used the notation $v + E \coloneqq\{ v+e:\: e\in E\}$ for $v\in\R^{2}$ and~$E\subseteq\R^{2}$.

If $\Haus(\Lambda^{n}) = s$, we simply set
\begin{displaymath}
\Gamma^{n}_{\sigma}\coloneqq \left\{ \begin{array}{ll}
\Lambda^{n}_{\sigma} & \text{if $\sigma\in[0, \Haus(\Lambda^{n})]$}\,,\\[1mm]
\Lambda^{n} \cup \big( (P_{n} - \gamma(s)) + \Gamma_{\sigma}\setminus\Gamma_{s }\big) & \text{if $\sigma\in (\Haus(\Lambda^{n}) , s + \delta]$}\,.
\end{array}\right.
\end{displaymath}
In both cases, we have that~$\{\Gamma^{n}_{\sigma}\}_{\sigma \in [0, s + \delta]}  \subseteq \mathcal{R}^{0,1}_{\eta}$,~$\Gamma^{n}_{\sigma} \to \Gamma_{\sigma}$ in the Hausdorff metric of sets for every~$\sigma\in[0,s + \delta]$, and~$\Gamma^{n}_{s}$ is of class~$C^{\infty}$ close to its tip.

The construction of~$f_{n}$ is trivial, since the set of functions in~$L^{2}(\Om;\R^{2})$ that vanish close to~$\gamma(s)$ is dense in~$L^{2}(\Om;\R^{2})$ w.r.t.~the $L^{2}$-norm. We only have to ensure that~$f_{n}$ is also null close to the tip of~$\Gamma^{n}_{s}$, which is still possible because of the Hausdorff convergence.

As for the elasticity tensor~$\C$, for every~$r>0$ we consider a cut off function~$\varphi_{r}$ in~$\B_{r}(\gamma(s))$ with~$\varphi_{r}(x)= \varphi_{r}(|x - \gamma(s)|)$, $\varphi_{r}=1$ in~$\overline{\B}_{r/2}(\gamma(s))$, and~$|\nabla \varphi_{r}|\leq C/r$ for some positive constant~$C$ independent of~$r$. Let us set~$\C_{s}\coloneqq \C(\gamma(s))$ and~$\C_{r}\coloneqq \varphi_{r} \C + (1-\varphi_{r})\C_{s}$. It is easy to see that~$\C_{r}\in C^{0,1}(\overline{\Om})$ with Lipschitz constant bounded by $C(\Lip(\C) +1)$. Hence,~$\C_{r} \rightharpoonup \C$ weakly* in~$W^{1,\infty}(\Om)$ as~$r\searrow 0$. To conclude, it is enough to choose a suitable sequence~$r_{n}\searrow 0$ in such a way that~$\C_{n}\coloneqq \C_{r_{n}}$ is constant close to the tip of~$\Gamma^{n}_{s}$. This is possible thanks to the Hausdorff convergence of~$\Gamma^{n}_{s}$ to~$\Gamma_{s}$.

The last part of the lemma is a trivial consequence of the above construction.
\end{proof}

\begin{proof}[Proof of Theorem~\ref{t.ERR}]

To prove~\eqref{e:finalERR}, we apply Lemma~\ref{l.approximation} to both~$\Gamma_{\sigma}$ and~$\hat{\Gamma}_{\sigma}$. Fixed $s\in(0,S)$ and $\delta>0$ small, let~$\{\Gamma^{n}_{\sigma}\}_{\sigma\in [0, s+\delta]}$, $\{\hat{\Gamma}_{\sigma}\}_{\{\sigma\in[0,s+\delta]}$,~$\C_{n}$, and~$f_{n}$ be as in Lemma~\ref{l.approximation}. By Corollary~\ref{c.continuityERR} we have that
\begin{eqnarray*}
& \displaystyle \lim_{n}\, \frac{\de\En(f_{n}, g, w,\C_{n};\Gamma^{n}_{\sigma})}{\de \sigma} \bigg|_{\sigma = s}= \frac{\de\En (f, g, w, \C;\Gamma_{ \sigma})}{\de \sigma} \bigg|_{\sigma = s} \,,\\
& \displaystyle  \lim_{n}\, \frac{\de\En(f_{n}, g, w,\C_{n}; \hat \Gamma^{n}_{\sigma})}{\de \sigma} \bigg|_{\sigma = s} = \frac{\de\En (f, g, w,\C; \hat{\Gamma}_{\sigma})}{\de \sigma} \bigg|_{\sigma = s}\,.
\end{eqnarray*}
Taking into account Proposition~\ref{p.localized}, we have that
\begin{displaymath}
\frac{\de\En(f_{n}, g, w,\C_{n};\Gamma^{n}_{\sigma})}{\de \sigma} \bigg|_{\sigma = s} = \frac{\de\En (f_{n}, g, w,\C_{n};\hat{\Gamma}^{n}_{\sigma})}{\de \sigma} \bigg|_{\sigma = s}\qquad\text{for every~$n\in \mathbb N$}
\end{displaymath}
and we deduce~\eqref{e:finalERR}.
\end{proof}

We are now in a position to give the precise definition of \emph{energy release rate} for a crack of the form~\eqref{crack}. We stress that this is now possible thanks to Theorem~\ref{t.ERR}. 

\begin{definition}\label{d.ERR}
Let~$\Gamma \in \mathcal{R}^{0,1}_{\eta}$, $ s\coloneqq \Haus(\Gamma)$, $f\in L^2(\Omega;\mathbb R^2)$, $g\in L^{2}(\partial_{S}\Om;\R^{2})$, $w\in H^{1}(\Om\setminus\Gamma_{0};\R^{2})$, and $\C\in C^{0,1}(\overline{\Om})$. Let $S>  s$ and let~$\{\Gamma_{\sigma}\}_{\sigma \in [0,S]} \subseteq \mathcal{R}^{0,1}_{\eta}$ be such~$\Gamma_{ s}=\Gamma$. We define the energy release~$\mathcal{G}(\Gamma)$ as
\begin{displaymath}
\mathcal{G}(\Gamma)\coloneqq - \frac{\de \En(\Gamma_{\sigma})}{\de \sigma}\bigg|_{\sigma = s} .
\end{displaymath} 
\end{definition}

\begin{remark}\label{r.ERR}
Definition~\ref{d.ERR}, stated for a curve~$\Gamma \in \mathcal{R}^{0,1}_{\eta}$, can be further generalized in order to consider general cracks in the class~$\adcr$. Indeed, given~$\Gamma\in\adcr$, it is enough to represent it as union of arcs of~$C^{1,1}$ curves~$\Gamma^{m}$, $m=1,\ldots, M$. In particular, each component belongs to~$\mathcal{R}^{0,1}_{\eta}$ and can be written as in~\eqref{crack}. Hence, for every~$m$ we define the $m$-th energy release rate~$\mathcal{G}^{m}(\Gamma)$ as in Definition~\ref{d.ERR} w.r.t.~variations of the sole component~$\Gamma^{m}$ of~$\Gamma$. The energy release rate will be in this case the vector
\begin{displaymath}
\mathcal{G}(\Gamma) \coloneqq (\mathcal{G}^{1}(\Gamma), \ldots, \mathcal{G}^{M}(\Gamma))\,.
\end{displaymath}
\end{remark}

\begin{remark}\label{r.ERR2}
We collect here the main properties of the energy release rate~$\G$.
\begin{itemize}
\item[$(a)$] $\G$ is continuous w.r.t.~the Hausdorff convergence of cracks $\Gamma\in\adcr$,  strong   convergence of volume forces~$f\in L^{2}(\Om;\R^{2})$, weak convergence of surface forces~$g\in L^{2}(\partial_{S}\Om;\R^{2})$, and convergence of Dirichlet boundary data~$w\in H^{1}(\Om \setminus \Gamma_{0};\R^{2})$;

\item[$(b)$] there exists a positive constant~$C= C(\C,\eta)$ such that for every $\Gamma\in\adcr$, every~$f\in L^{2}(\Om;\R^{2})$, every~$g\in L^{2}(\partial_{S}\Om;\R^{2})$, and every~$w\in H^{1}(\Om \setminus \Gamma_{0};\R^{2})$,
\begin{displaymath}
0\leq \G(\Gamma) \leq C \|u\|_{H^{1}}^{2} + \|f\|_{2}\|u\|_{H^{1}}\,,
\end{displaymath}
where~$u\in H^{1}(\Om\setminus\Gamma;\R^{2})$ is the solution of~\eqref{pb:min} with data~$\Gamma$,~$f$,~$g$,~$w$, and $\C$.
\end{itemize}
We will make use of these two properties in the proofs of Proposition~\ref{prop:viscous} and Theorems~\ref{thm:bv} and~\ref{thm:param}.
\end{remark}

\section{Vanishing viscosity evolutions}\label{sec:vve}

In this section we focus on the proofs of existence of a viscous evolution~$\Gamma_{\eps}$ (see Proposition~\ref{prop:viscous}) and of a balanced viscosity evolution~$\Gamma$ (Theorem~\ref{thm:bv}), the latter obtained as limit of~$\Gamma_{\eps}$ as the viscosity parameter~$\eps$ tends to 0. To this end, we follow the method employed in a wide literature on rate-independent processes~\cite{MR3380972}. However, we point out that the abstract results of~\cite{MR2525194,MR2887927,MR3531671} do not directly apply to our setting.  

Since the problem we analyze depends explicitly on time~$t$ through the applied loads~$f$,~$g$, and~$w$, from now on we denote with~$\G(t; \Gamma)$ the energy release rate defined in Definition~\ref{d.ERR} and Remark~\ref{r.ERR} for a crack~$\Gamma\in \adcr$ at time~$t\in[0,T]$.

As anticipated in Section~\ref{s:model}, the proofs of Proposition~\ref{prop:viscous} and of Theorem~\ref{thm:bv} are based on a time-discretization procedure. Let the initial crack~$\Gamma_{0}\in\adcr^{0}$ and the viscosity parameter~$\eps > 0$ be fixed, and let us set~$l^{m}_{0}\coloneqq \Haus(\Gamma^{m}_{0})$, where $\Gamma_{0} = \bigcup_{m=1}^{M} \Gamma^{m}_{0}$, according to Definition~\ref{defin:adcr}. For every~$k\in\mathbb{N}$ we fix a partition~$\{t_{k,i}\}_{i=0}^{k}$ of the time interval~$[0,T]$ as in~\eqref{defin:time-disc}. For $i=0$ we set~$\Gamma_{\eps,k,0}\coloneqq \Gamma_{0}$. For $i\in\{1, \ldots, k\}$ we denote with~$\Gamma_{\eps, k, i}$ a minimizer of the incremental minimum problem~\eqref{pb:increm-min-eps}, whose existence is provided by Corollary~\ref{cor:disc-exist}. Recalling the conventions of Definition~\ref{defin:adcr2}, we write~$\Gamma_{\eps,k,i} = \bigcup_{m=1}^{M} \Gamma_{\eps, k, i}^{m}$, we set~$l^{m}_{\eps, k, i}\coloneqq \Haus(\Gamma^{m}_{\eps, k, i})$, and we denote with~$P^{m}_{\eps, k, i}$ the tip of~$\Gamma^{m}_{\eps, k, i}$. Furthermore, we define the interpolation functions 
\begin{alignat*}{3}
& \displaystyle l^{m}_{\eps, k} (t) \coloneqq l^{m}_{\eps, k, i-1} + (l^{m}_{\eps, k, i} - l^{m}_{\eps, k, i-1}) \, \frac{ t - t_{k,i-1} }{t_{k,i} - t_{k,i-1}} \,,\\[1mm]
& \displaystyle \G^{m}_{\eps, k}(t) \coloneqq \G^{m}(t_{k,i} ; \Gamma_{\eps, k, i})\,, \quad \G_{\eps, k}(t)\coloneqq \G(t_{k,i};\Gamma_{\eps, k, i})  \,, \\[1mm]
& \displaystyle \Gamma_{\eps, k} (t)\coloneqq \Gamma_{\eps, k, i}\,,\quad P^{m}_{\eps, k}(t) \coloneqq P^{m}_{\eps, k, i} \,, \quad u_{\eps, k}(t)\coloneqq u^{\eps}_{k,i}  \,, \\
& \displaystyle t_{k}(t) \coloneqq t_{k,i}\,,\quad f_{k} (t) \coloneqq f(t_{k,i}) \,, \quad g_{k} (t) \coloneqq g(t_{k,i})  & \qquad \text{for $t\in(t_{k,i-1}, t_{k,i}]$} \,,\\
& \displaystyle \underline{\Gamma}_{\eps, k}(t)\coloneqq \Gamma_{\eps, k, i-1} \,,\qquad \underline{u}_{\eps, k}(t)\coloneqq u^{\eps}_{k,i}& \qquad \text{for $t\in[t_{k,i-1}, t_{k,i})$}\,,
\end{alignat*}
where we denoted with~$u^{\eps}_{k,i}$ the function~$u(t_{k,i};\Gamma_{\eps, k, i}) \in H^{1}(\Om\setminus\Gamma_{\eps, k, i})$.

In the following proposition we state a time discrete version of the Griffith's criterion~(G1)$_{\eps}$--(G3)$_{\eps}$.

\begin{proposition}\label{prop:discrete-Griffith}
For every~$\eps>0$, every~$k\in\mathbb{N}$, every~$m\in\{1,\ldots, M\}$, and a.e.~$t\in(0,T]$ it holds:
\begin{itemize}
\item[(G1)$_{k}$] $\dot{l}^{m}_{\eps, k}(t) \geq 0$;

\item[(G2)$_{k}$] $\kappa(P^{m}_{\eps, k} (t) ) -\G^{m}_{\eps, k} ( t ) +  \eps \dot{l}^{m}_{\eps, k}(t) \geq 0$;

\item[(G3)$_{k}$] $ \dot{l}^{m}_{\eps, k}(t) ( \kappa( P^{m}_{\eps, k} (t) ) - \G^{m}_{\eps, k} ( t ) + \eps \dot{l}_{\eps, k}(t))  = 0$.
\end{itemize}
\end{proposition}

\begin{proof}
By construction,~$l^{m}_{\eps, k}$ is a non-decreasing function, so that~(G1)$_{k}$ is clearly satisfied. In order to show~(G2)$_{k}$--(G3)$_{k}$ we take into account the minimality of~$\Gamma_{\eps, k, i}$. Let us fix~$i\in\{1,\ldots\, k \}$. For every~$\overline{m}\in \{1,\ldots, M\}$, let~$\Gamma_{\overline{m}}\in \adcr$ be such that~$\Gamma_{ \overline{m}} = \bigcup_{m\neq \overline{m}} \Gamma_{\eps, k, i}^{m} \cup \Lambda^{\overline{m}}$ with~$\Lambda^{\overline m} \supseteq \Gamma^{\overline{m}}_{\eps, k, i}$, and let us set~$\lambda\coloneqq \Haus(\Lambda^{\overline{m}}) \geq l^{\overline{m}}_{\eps, k, i}$. Then,
\begin{displaymath}
\begin{split}
\En (t_{k,i};\Gamma_{\eps, k, i}) & + \int_{\Gamma_{\eps, k, i}} \!\!\!\!\!\! \kappa\,\di \Haus +  \frac{\eps}{2} \sum_{m=1}^{M}\frac{\Haus(\Gamma^{m}_{\eps, k, i}\setdiff\Ga^{m}_{\eps,k,i-1})^2}{t_{k,i}-t_{k,i-1}} \\
&  \leq \En (t_{k,i};\Gamma_{ \overline{m}}) + \int_{\Gamma_{\overline{m}}} \!\! \kappa\,\di \Haus  +  \frac{\eps}{2} \sum_{m\neq \overline{m}} \frac{\Haus(\Gamma^{m}_{\eps, k, i}\setdiff\Ga^{m}_{\eps,k,i-1})^2}{t_{k,i}-t_{k,i-1}} +  \frac{\eps}{2} \frac{\Haus(\Lambda^{\overline{m}}\setdiff\Ga^{\overline{m}}_{\eps,k,i-1})^2}{t_{k,i}-t_{k,i-1}}\,,
\end{split}
\end{displaymath}
which implies
\begin{equation}\label{boh1}
\En (t_{k,i};\Gamma_{\eps, k, i}) +  \int_{\Gamma^{\overline{m}}_{\eps, k, i}}\kappa\,\di \Haus  +  \frac{\eps}{2} \frac{(l^{\overline{m}}_{\eps, k, i} - l^{\overline{m}}_{\eps, k, i-1})^2}{t_{k,i}-t_{k,i-1}} \leq \En (t_{k,i};\Gamma_{\overline{m}}) +  \int_{\Lambda^{\overline{m}}} \kappa\,\di \Haus  +  \frac{\eps}{2}  \frac{( \lambda - l^{\overline{m}}_{\eps, k, i-1})^2}{t_{k,i}-t_{k,i-1}}\,.
\end{equation}
We divide~\eqref{boh1} by~$\lambda - l^{\overline{m}}_{\eps, k, i}$ and pass to the limit as~$\lambda \to l^{\overline{m}}_{\eps, k, i}$, obtaining~(G2)$_{k}$ as a consequence of~(H4) and of Definition~\ref{d.ERR}.
If, moreover, $\Gamma^{\overline{m}}_{\eps, k, i-1}\subsetneq \Gamma^{\overline{m}}_{\eps, k, i}$, we can consider as a competitor a set~$\Gamma_{\overline{m}}\in \adcr$ as above, with~$\Gamma^{\overline{m}}_{\eps, k, i-1}\subseteq \Lambda^{\overline{m}} \subseteq \Gamma^{\overline{m}}_{\eps, k, i}$, so that $ l^{\overline{m}}_{\eps, k, i-1}\leq \lambda \leq l^{\overline{m}}_{\eps, k, i}$. Repeating the above computation we obtain
\[
\kappa(P^{\overline{m}}_{\eps, k, i}) - \G^{\overline{m}}( t_{k, i}; \Gamma^{\overline{m}}_{\eps, k, i}) + \eps \, \frac{l^{\overline{m}}_{\eps, k, i} - l^{\overline{m}}_{\eps, k, i-1}}{t_{k,i}- t_{k, i-1}} =0 \,.
\]
This concludes the proof of~(G3)$_{k}$.
\end{proof}

We now show an a priori bound on~$l_{\eps, k}$ and on~$u_{\eps, k}$.

\begin{proposition}\label{prop:bounds}
The following facts hold:
\begin{itemize}
\item[$(a)$] there exist two positive constants~$c$~and~$C$ independent of~$\eps$,~$k$, and~$i$ such that for every~$\eps>0$, every~$k\in\mathbb{N}$, and every~$t\in[0,T]$,
%
%
%
 \begin{equation}\label{e:bound1.2} 
 \begin{split}
& \hspace{-5em}  \frac{\eps}{2} \sum_{m=1}^{M} \int_{0}^{t_{k}(t)}  |\dot{l}^{m}_{\eps, k}(\tau)|^{2}\,\di\tau   + c \, (\| \e u_{\eps, k}(t) \|^2_{2} - \|u_{\eps, k}(t)\|_{H^{1}})  \\
  \leq\ & \F(0;\Gamma_{0}) + \int_{0}^{t_k(t)} \int_{\Om}\C \e \underline{u}_{\eps, k}(\tau) : \e \dot{w}(\tau)\,\di x \, \di \tau + C\sum_{i=1}^{k} \| w_{k,i} - w_{k,i-1} \|_{H^{1}}^{2}\\
&- \int_{0}^{t_k(t)} \int_{\Om} \dot{f} (\tau) \cdot \underline{u}_{\eps, k} (\tau)\, \di x \, \di \tau  
-\int_{0}^{t_k(t)}\int_{\Om} f_{k}(\tau) \cdot \dot{w}(\tau)\,\di x\,\di \tau \\
&- \int_{0}^{t_k(t)} \int_{\partial_{S}\Om} \dot{g}(\tau)\cdot \underline{u}_{\eps, k}(\tau)\,\di \Haus\,\di\tau 
- \int_{0}^{t_k(t)}\int_{\partial_{S}\Om} g_{k}(\tau)\cdot \dot{w}(\tau)\,\di\Haus\,\di \tau \,;
\end{split}
\end{equation}

\item[$(b)$] for every~$\eps>0$, along a suitable (not relabeled) subsequence, $\|u_{\eps,k}(t)\|_{2}$ and~$\|\nabla{u}_{\eps, k}(t)\|_{2}$ are bounded uniformly w.r.t.~$t\in[0,T]$ and~$k \in \mathbb{N}$;

\item[$(c)$]for every~$\eps>0$, along a suitable (not relabeled) subsequence,~$\eps \|\dot{l}^{m}_{\eps, k}\|^{2}_{2}$ is bounded uniformly w.r.t.~$k\in\mathbb{N}$  and~$m\in\{1,\ldots,M\}$.
\end{itemize}
\end{proposition}

\begin{proof}
For the sake of simplicity, let us denote with~$w_{k,i}$, $f_{k,i}$, and~$g_{k,i}$ the functions~$w(t_{k,i})$,~$f(t_{k,i})$, and~$g(t_{k,i})$, respectively.

By definition of~$u_{k,i}^{\eps}$, by hypothesis~(H3), and by the regularity of the data of the problem~$f$,~$g$, and~$w$, we have that
\begin{equation}\label{e:bound1.1} 
\begin{split}
C_{1} (\| \e u^{\eps}_{k,i}\|_{2}^{2} - \|u^{\eps}_{k,i}\|_{H^{1}} ) & \leq \En(t_{k,i};\Gamma_{\eps, k, i}) \\
&  \leq \frac{1}{2}\int_{\Om} \C\e w_{k,i}:\e w_{k,i}\,\di x - \int_{\Om} f_{k,i}\cdot w_{k,i}\,\di x - \int_{\partial_{S}\Om} g_{k,i}\cdot w_{k,i}\,\di \Haus \leq C_{2}\,,
\end{split}
\end{equation}
for some positive constants~$C_{1}$ and~$C_{2}$ depending only on~$f$,~$g$,~$w$, and~$\C$. 

Since, for every~$\eps$ and~$k$, the set function~$\Gamma_{\eps, k}\colon[0,T]\to\adcr$ is nondecreasing, we have that~$u_{k,i}^{\eps}\in H^{1}(\Om\setminus\Gamma_{\eps, k}(T);\R^{2})$. By definition of the class~$\adcr$, the curves~$\Gamma_{\eps, k}(T)$ have bounded length uniformly w.r.t.~$\eps$ and~$k$. Hence, we may assume that, up to a not relabeled subsequence,~$\Gamma_{\eps, k}(T) \to \widehat{\Gamma}_{\eps}\in \adcr$ in the Hausdorff metric of sets as $k\to\infty$. We are therefore in a position to apply Proposition~\ref{p.korn2} to~$\Gamma_{\eps, k}(T)$,~$\widehat{\Gamma}_{\eps}$, and~$u^{\eps}_{k,i}$, which, together with~\eqref{e:bound1.1}, implies~(b).  

By definition of~$\Gamma_{\eps, k, i}$ and of the energy~$\En(t_{k,i}; \Gamma)$ we have that
\begin{displaymath}
\begin{split}
& \En(t_{k,i};\Gamma_{\eps,k,i}) + \int_{\Gamma_{\eps, k, i}} \!\!\!\!\!\! \kappa\,\di\Haus + \frac{\eps}{2}\sum_{m=1}^{M} \frac{\Haus(\Gamma^{m}_{\eps, k, i} \setminus \Gamma^{m}_{\eps, k, i-1})^{2} }{ t_{k,i} - t_{k,i-1} } \leq \En(t_{k,i}; \Gamma_{\eps, k, i-1}) + \int_{\Gamma_{\eps, k, i-1}} \!\!\!\!\!\!\!\! \kappa\,\di \Haus \\
& \leq \frac{1}{2}\int_{\Om} \C\e (u^{\eps}_{k,i-1} + w_{k,i}- w_{k, i-1}) : \e (u^{\eps}_{k,i-1} + w_{k,i}- w_{k, i-1})\,\di x \\
 & \qquad - \int_{\Om} f_{k,i}\cdot  (u^{\eps}_{k,i-1} + w_{k,i}- w_{k, i-1})\,\di x - \int_{\partial_{S} \Om} g_{k,i}\cdot  (u^{\eps}_{k,i-1} + w_{k,i}- w_{k, i-1})\,\di \Haus + \int_{\Gamma_{\eps, k, i-1}} \!\!\!\!\!\!\!\! \kappa \,\di\Haus \\
& = \En (t_{k,i-1}; \Gamma_{\eps, k, i-1}) + \int_{\Om}\C\e u^{\eps}_{k,i-1}: \e (w_{k,i}- w_{k, i-1})\,\di x + \frac{1}{2}\int_{\Om} \C\e (w_{k,i}-w_{k,i-1}) : \e(w_{k,i}- w_{k,i-1})\,\di x \\
 & \qquad - \int_{\Om} f_{k,i} \cdot (w_{k,i} - w_{k,i-1})\,\di x -\int_{\Om} (f_{k,i} - f_{k,i-1})\cdot u^{\eps}_{k,i-1}\,\di x -\int_{\partial_S\Om} g_{k,i}\cdot (w_{k,i}-w_{k,i-1})\,\di \Haus \\
 & \qquad  - \int_{\partial_{S}\Om} (g_{k,i}- g_{k,i-1})\cdot u^{\eps}_{k,i-1}\,\di\Haus+\int_{\Gamma_{\eps,k,i-1}} \!\!\!\!\!\!\! \kappa\,\de\Haus\,.
\end{split}
\end{displaymath}
Iterating the above chain of inequalities for~$i\in\{1,\ldots, k\}$ and using~(H2) we deduce~\eqref{e:bound1.2}, which, together with~(b), implies~(c).
\end{proof}

In the following proposition we discuss  the properties of  the limit of the sequence~$\Gamma_{\eps, k}$ as~$k\to\infty$.

\begin{proposition} \label{prop:conv-n}
For every $\eps>0$ there exists a subsequence (not relabeled)
of $\Ga_{\eps,k}$ and a set function $t\mapsto\Ga_\eps(t)\in\adcr$
such that $\Ga^{m}_{\eps,k}(t)$ converges to $\Ga^{m}_\eps(t)$ in the Hausdorff metric for every $t\in[0,T]$ and every~$m\in\{1,\ldots,M\}$, and
\begin{itemize}
\item[$(a)$] $\Ga_\eps$ is nondecreasing in time;

\item[$(b)$] $l^{m}_{\eps,k}\wto l^{m}_\eps$ weakly in~$H^1(0,T)$ and $l^{m}_{\eps,k}(t) \to l^{m}_\eps (t)$ for every $t\in[0,T]$ and every~$m$, where $l^{m}_\eps(t)\coloneqq \Haus (\Ga^{m}_\eps(t))$;

\item[$(c)$]   $\nabla u_{\eps,k}(t)\to\nabla u_\eps(t)$ strongly in $L^2(\Om;\mathbb{M}^2)$ for every $t\in[0,T]$, where $u_\eps(t) \coloneqq u(t;\Ga_\eps(t))$;

\item[$(d)$]  $\G_{\eps,k}(t)\to \G_\eps(t)$ for every $t\in[0,T]$, where $\G_\eps(t)\coloneqq \G(t;\Ga_\eps(t))$;

\item[$(e)$] $\G_{\eps,k} \to \G_\eps$ in~$L^{2}(0,T)$.
\end{itemize}
Moreover, along a suitable (not relabeled) subsequence, we have
\begin{itemize}
\item[$(f)$] $\eps\|{\dot l^{m}_\eps}\|_{2}^2$ is uniformly bounded in~$\eps$ for every~$m\in\{1,\ldots,M\}$;

\item[$(g)$] $\|{\nabla u_\eps(t)}\|_2$ is uniformly bounded in~$\eps$ and~$t$.
\end{itemize}
\end{proposition}

\begin{proof}
 For brevity, in the following we will not relabel  subsequences.  
For $\eps>0$ let us consider the subsequence~$\Gamma_{\eps,k}$ detected in~(b) and~(c) of Proposition~\ref{prop:bounds}. Since~$\Gamma_{\eps,k}$ is a sequence of increasing set functions with uniformly bounded length, there exists a nondecreasing set function~$\Gamma_{\eps}\colon [0,T]\to \adcr$ such that, up to a further subsequence,~$\Gamma_{\eps, k}(t)$ converges to~$\Gamma_{\eps}(t)$ in the Hausdorff metric of sets for every~$t\in[0,T]$. Hence, for every~$m\in\{1,\ldots,M\}$ and every~$t\in[0,T]$ it holds $\Gamma^{m}_{\eps,k}(t) \to \Gamma^{m}_{\eps}(t)$.

For~$\eps>0$ fixed, by Proposition~\ref{prop:bounds} we have that~$l^{m}_{\eps, k} \in H^{1}(0,T)$ is bounded w.r.t.~$k$ and~$m$. Therefore, for every~$m\in\{1,\ldots,M\}$ the sequence~$l^{m}_{\eps,k}$ converges weakly in~$H^{1}(0,T)$ to a nondecreasing function~$l^{m}_{\eps}$. Up to a further subsequence, we may assume that~$l^{m}_{\eps,k}(t) \to l^{m}_{\eps}(t)$ for every~$t\in[0,T]$ and every~$m$. In particular, $l^{m}_{\eps}(t) = \Haus(\Gamma^{m}_{\eps}(t))$, so that~(b) is proven. We also notice that, because of the continuity of~$l^{m}_{\eps}$, we have that~$\underline{\Gamma}_{\eps, k}(t)\to \Gamma_{\eps}(t)$ in the Hausdorff metric as $k\to\infty$.

The $L^{2}$-convergence of~$\nabla u_{\eps,k}(t)$ to~$ \nabla u_\eps(t)$ is a consequence of the convergence of~$\Gamma_{\eps, k}(t)$ to~$\Gamma_{\eps}(t)$ and of Lemma~\ref{l.continuity}. In a similar way, since~$\underline{\Gamma}_{\eps, k}(t)$ converges to~$\Gamma_{\eps}(t)$,~$\nabla \underline{u}_{\eps, k}(t) \to \nabla u_{\eps}(t)$ in~$L^{2}(\Om;\mathbb{M}^{2})$. Moreover, by Remark~\ref{r.ERR2} we have that~$\G_{\eps, k}(t) \to \G(t ; \Gamma_{\eps}(t)) =: \G_{\eps}(t)$ for every~$t\in[0,T]$, so that~(d) holds. Being~$\|\nabla{u}_{\eps, k}(t)\|_{2}$ and~$\|u_{\eps, k}(t)\|_{2}$ bounded uniformly w.r.t.~$t$ and~$k$, again by Remark~\ref{r.ERR2} we infer that~$\G_{\eps, k}(t)$ is bounded, so that~$\G_{\eps, k}\to \G_{\eps}$ in~$L^{2}(0,T)$ and~(e) is concluded.

In order to prove~(f) and~(g), we employ Proposition~\ref{prop:bounds}, obtaining
\begin{equation}\label{e:limiteps}
\| \e u_{\eps}(t)\|_{2}^2 - \|u_{\eps}(t)\|_{H^1} \leq C \qquad\text{for every~$\eps>0$ and every~$t\in[0,T]$}\,,
\end{equation}
where~$C>0$ is independent of~$t$ and~$\eps$. Arguing as in the proof of Proposition~\ref{prop:bounds}, we have that~$u_{\eps}(t)\in H^{1}(\Om\setminus\Gamma_{\eps}(T);\R^{2})$ for every~$t\in[0,T]$. Since~$\Gamma_{\eps}(T)\in \adcr$ has a uniformly bounded length, we may assume that, up to a subsequence, $\Gamma_{\eps}(T)\to \widehat{\Gamma} \in\adcr$ in the Hausdorff metric of sets. Thus, we can apply Proposition~\ref{p.korn2} to~$\Gamma_{\eps}(T)$,~$\widehat{\Gamma}$, and~$u_{\eps}(t)$, to deduce from~\eqref{e:limiteps} that~$\|\nabla{u}_{\eps}(t)\|_{2}$ is bounded uniformly w.r.t.~$\eps$ and~$t$, so that~(g) holds.

Finally, we pass to the liminf in~\eqref{e:bound1.2} for~$t=T$, obtaining
\[
\begin{split}
  \frac{\eps}{2} \sum_{m=1}^{M} \int_{0}^{T} & |\dot{l}^{m}_{\eps}(t)|^{2}\,\di t   + C_{1}(\| \e u_{\eps}(T) \|_{2} - \|u_{\eps}(T)\|_{2})  \\
&  \leq \En(0;\Gamma_{0}) + \int_{\Gamma_{0}} \kappa\,\di \Haus + \int_{0}^{T} \int_{\Om}\C \e u_{\eps}(t) : \e \dot{w}(t)\,\di x \, \di t - \int_{0}^{T} \int_{\Om} \dot{f} (t) \cdot u_{\eps} (t)\, \di x \, \di t  \\
& \qquad  -\int_{0}^{T}\int_{\Om} f (t) \cdot \dot{w}(t)\,\di x\,\di t - \int_{0}^{T} \int_{\partial_{S}\Om} \dot{g}(t)\cdot u_{\eps}(t)\,\di \Haus\,\di t  - \int_{0}^{T}\int_{\partial_{S}\Om} g (t) \cdot \dot{w}(t)\,\di\Haus\,\di t\,.
\end{split}
\]
By the boundedness of~$\|\nabla{u}_{\eps}(t)\|_{2}$ and of~$\|u_{\eps}(t)\|_{2}$ we immediately get~(f), and the proof is thus concluded.
\end{proof}

We are now in a position to prove Proposition~\ref{prop:viscous}.

\begin{proof}[Proof of Proposition~\ref{prop:viscous}]
Let~$\Gamma_{\eps}$,~$l^{m}_{\eps}$, and~$\G^{m}_{\eps}$ be the functions determined in Proposition~\ref{prop:conv-n}. Since~$l^{m}_{\eps}$ is nondecreasing,~(G1)$_{\eps}$ is satisfied. In order to prove~(G2)$_{\eps}$ let us consider~$\psi\in L^{2}(0,T)$ with~$\psi\geq0$. By~(G2)$_{k}$ we have
\begin{equation}\label{e:bah}
\int_{0}^{T} (\kappa(P^{m}_{\eps,k}(t)) - \G^{m}_{\eps,k} (t) + \eps \dot{l}^{m}_{\eps,k}(t)) \psi(t)\,\di t\geq 0\,.
\end{equation}
From the Hausdorff convergence of~$\Gamma^{m}_{\eps,k}(t)$ to~$\Gamma^{m}_{\eps}(t)$ it follows that~$P^{m}_{\eps,k}(t)\to P^{m}_{\eps}(t)$ for every~$t\in[0,T]$ and every~$m$, where~$P^{m}_{\eps}(t)$ stands for the tip of~$\Gamma^{m}_{\eps}(t)$.
By hypothesis~(H4) we have that~$\kappa(P^{m}_{\eps, k}) \to \kappa(P^{m}_{\eps})$ in~$L^{2}(0,T)$. Hence, passing to the limit in~\eqref{e:bah} as $k\to\infty$ and taking into account~(e) of Proposition~\ref{prop:conv-n} we get
\begin{displaymath}
\int_{0}^{T} (\kappa(P^{m}_{\eps}(t)) - \G^{m}_{\eps} (t) + \eps \dot{l}^{m}_{\eps}(t)) \psi(t)\,\di t\geq 0\,.
\end{displaymath} 
By the arbitrariness of~$\psi\in L^{2}(0,T)$, $\psi\geq0$, we infer~(G2)$_{\eps}$.

As for~(G3)$_{\eps}$, we integrate~(G3)$_{k}$ over~$[0,T]$ and pass to the liminf as~$k\to\infty$. By~(b) and~(e) of Proposition~\ref{prop:conv-n} and by the convergence of~$\kappa(P^{m}_{\eps, k})$ to~$\kappa(P^{m}_{\eps})$ we obtain
\[
\int_{0}^{T} \dot{l}^{m}_{\eps}(t) (\kappa(P^{m}_{\eps}(t)) - \G^{m}_{\eps}(t) + \eps \dot{l}^{m}_{\eps}(t))\,\di t\leq 0\,.
\]
Combining the previous inequality with~(G1)$_{\eps}$ and~(G2)$_{\eps}$ we deduce~(G3)$_{\eps}$. Finally, the uniform boundedness of~$\eps \|\dot{l}^{m}_{\eps}\|_{2}^{2}$ has been stated in~(f) of Proposition~\ref{prop:conv-n}.
\end{proof}

\begin{remark}\label{r.energybalance}
Let~$\Gamma_{\eps}$ be as in Proposition~\ref{prop:viscous}. Then, for every~$t\in [0,T]$ it holds
\begin{equation}\label{e:vbalance}
\begin{split}
\F (t; \Gamma_{\eps}(t)) = \ & \F(0;\Gamma_{0}) - \sum_{m=1}^M \int_{0}^{t} (\G^{m}_{\eps}(\tau) - \kappa(P^{m}_{\eps})) \dot{l}_{\eps}^{m}(\tau)\,\di\tau + \int_{0}^{t}\int_{\Om}\C\e u_{\eps}(\tau) : \e\dot{w}(\tau)\,\di x\,\di \tau \\
& - \int_{0}^{t}\int_{\Om}\dot{f}(\tau)\cdot u_{\eps}(\tau)\,\di x\,\di\tau - \int_{0}^{t}\int_{\Om} f(\tau)\cdot \dot{w}(\tau)\,\di x\,\di \tau \\
& - \int_{0}^{t} \int_{\partial_{S}\Om} \dot g(\tau)\cdot u(\tau)\,\di \Haus \,\di \tau- \int_{0}^{t}\int_{\partial_{S}\Om} g(\tau)\cdot\dot{w}(\tau)\,\di \Haus \,\di\tau\,.
\end{split}
\end{equation}
Indeed, being~$l_{\eps}\in H^{1}([0,T];\R^{M})$, the function~$t\mapsto \F(t,\Gamma_{\eps}(t)) = \En(t;\Gamma_{\eps}(t)) + \mathcal{K}(\Gamma_{\eps}(t))$ belongs to~$H^{1}(0,T)$ with
\begin{displaymath}
\frac{\di}{\di t} \F(t;\Gamma_{\eps}(t)) = \partial_{t} \En(t; \Gamma_{\eps}(t)) - \sum_{m=1}^M
\left[ \G^{m}_\eps(t)- \kappa(P^{m}_{\eps}(t)) \right]\,\dot{l}^{m}_{\eps}(t)\qquad\text{for a.e.~$t\in[0,T]$}\,.
\end{displaymath}
\end{remark}

We conclude with the proof of Theorem~\ref{thm:bv}.

\begin{proof}[Proof of Theorem~\ref{thm:bv}]
For~$\eps>0$ and~$m\in\{1,\ldots,M\}$ let~$\Gamma_{\eps}$,~$\Gamma^{m}_{\eps}$, and~$l^{m}_{\eps}$ be the viscous evolutions determined in Proposition~\ref{prop:viscous}.  Let us consider, without relabeling, the $\eps$-subsequence satisfying~(f) and~(g) of Proposition~\ref{prop:conv-n}.  Since~$\Gamma_{\eps}$ is a sequence of nondecreasing set functions and $\Haus(\Gamma_{\eps}(t))$ is uniformly bounded w.r.t.~$t\in[0,T]$ and~$\eps>0$, there exists a nondecreasing set function~$\Gamma\colon [0,T]\to \adcr$ such that~$\Gamma_\eps(t) \to \Gamma(t)$ in the Hausdorff metric of sets for every~$t\in[0,T]$. In particular,~$\Gamma^{m}_{\eps}(t) \to \Gamma^{m}(t)$ for every~$t$ and every~$m\in\{1,\ldots, M\}$, where~$\Gamma(t)= \bigcup_{m=1}^{M} \Gamma^{m}(t)$. Moreover, being~$l^{m}_{\eps}$ a sequence of bounded nondecreasing functions, we may assume that, up to a further subsequence, $l^{m}_{\eps}(t) \to l^{m}(t)$ for every~$t\in[0,T]$ and~$l^{m}_{\eps}\to l^{m}$ in~$L^{2}(0,T)$. In particular,~$l^{m}(t) = \Haus(\Gamma^{m}(t))$ and~(G1) is proven.

In order to show~(G2), let us consider~$\psi\in L^{2}(0,T)$ with~$\psi\geq0$. In view of~(G2)$_{\eps}$ we have
\begin{equation}\label{e:al}
\int_{0}^{T} \big (\kappa(P^{m}_{\eps}(t) )- \G^{m}_{\eps}(t) + \eps \dot{l}_{\eps}^{m}(t)\big ) \psi(t)\,\di t \geq0\,.
\end{equation}
Since~$\Gamma^{m}_{\eps}(t) \to \Gamma^{m}(t)$, we have that~$P^{m}_{\eps}(t) \to P^{m}(t)$ for every~$t$ and every~$m$, where~$P^{m}(t)$ is the tip of~$\Gamma^m(t)$. Thus, by hypothesis~(H4) we get that~$\kappa(P^{m}_{\eps}) \to \kappa(P^{m})$ in~$L^{2}(0,T)$ for every~$m$. From~(e) and~(f) of Proposition~\ref{prop:conv-n} we deduce that~$\eps \dot{l}^{m}_{\eps}\to 0$ and~$\G^{m}_{\eps} \to \G^{m}$ in~$L^{2}(0,T)$. Hence, passing to the limit in~\eqref{e:al} we get
\[
\int_{0}^{T} \big (\kappa(P^{m}(t) )- \G^{m}(t) \big ) \psi(t)\,\di t \geq0 \qquad \text{for every $\psi\in L^{2}(0,T)$, $\psi\geq0$}\,.
\]
As a consequence,~$\kappa(P^{m}(t) )- \G^{m}(t)\geq 0$ for a.e.~$t\in[0,T]$. By continuity, this inequality holds in all the continuity points of~$\Gamma^{m}(t)$. Hence,~(G2) is proven.

As for~(G3), we integrate~(G3)$_{\eps}$ over the interval~$[0,T]$ and notice that the term~$\eps(\dot{l}^{m}_{\eps})^{2}$ is positive, so that
\[
\int_{0}^{T}\dot{l}^{m}_{\eps}(t) ( \kappa(P^{m}_{\eps}(t)) - \G^{m}_{\eps}(t))\,\di t\leq 0\,.
\]
Passing to the limit in the previous inequality we get
\begin{equation}\label{e:bah4}
\int_{0}^{T}\dot{l}^{m} (t) ( \kappa(P^{m} (t)) - \G^{m} (t))\,\di t\leq 0\,.
\end{equation}
Combining~\eqref{e:bah4} with~(G1) and~(G2) we deduce~(G3).
\end{proof}

\section{Parametrized evolutions}\label{s:parametric}

This section is devoted to the proof of Theorem~\ref{thm:param}. The strategy is to perform a change of variables which transforms the lengths~$l^{m}_{\eps}$ obtained in Proposition~\ref{prop:viscous} in absolutely continuous functions. Roughly speaking, this is done by a parametrization of time on the jump points of the viscous solution~$l^{m}_{\varepsilon}$. 

Let us fix the sequence~$\eps\to 0$ determined in Proposition~\ref{prop:conv-n} and Theorem~\ref{thm:bv}. For $t\in[0,T]$ we set
\begin{equation}\label{II-81}
\sigma_{\varepsilon}(t) \coloneqq  t+\sum_{m=1}^{M}(l^{m}_{\varepsilon}(t)-l_{0}^{m})\,.
\end{equation}
Thanks to the properties of~$l^{m}_{\varepsilon}$  (see Proposition~\ref{prop:viscous}),~$\sigma_{\varepsilon}$ is strictly increasing, continuous, and $\dot{\sigma}_{\varepsilon}(t)\geq1$ for every $\varepsilon>0$ and a.e.~$t\in[0,T]$. Therefore,~$\sigma_{\eps}$ is invertible and we denote its inverse with~$\tilde{t}_{\varepsilon}\colon [0, \sigma_{\eps}(T)]\to [0, T]$. We deduce that~$\tilde{t}_{\varepsilon}$ is strictly increasing, continuous, and $0<\tilde{t}'_{\varepsilon}(\sigma)\leq1$ for every $\varepsilon>0$ and a.e.~$\sigma\in[0, \sigma_{\eps}(T) ]$, where the symbol $'$ stands for the derivative with respect to~$\sigma$.

For $m=1,\ldots,M$ and $\sigma\in[0, \sigma_{\eps}(T)]$, we set
\begin{eqnarray*}
&& \displaystyle \tilde{l}^{m}_{\varepsilon}(\sigma) \coloneqq l^{m}_{\varepsilon}(\tilde{t}_{\varepsilon}(\sigma))\,,\qquad \tilde{l}_{\varepsilon}(\sigma) \coloneqq (\tilde{l}^{1}_{\varepsilon}(\sigma),\ldots,\tilde{l}^{M}_{\varepsilon}(\sigma))\,,\qquad  \tilde{l}_{\varepsilon}'(\sigma) \coloneqq ((\tilde{l}_{\varepsilon}^{1})'(\sigma),\ldots,(\tilde{l}_{\varepsilon}^{M})'(\sigma))\,,\\
&& \widetilde{\Gamma}_{\eps} (\sigma )\coloneqq \Gamma_{\eps}(\tilde{t}_{\eps}(\sigma))\,,\qquad \widetilde \Gamma^{m}_{\eps}(\sigma) \coloneqq \Gamma^{m}_{\eps}(\tilde{t}_{\eps}(\sigma))\,,\qquad \widetilde{P}^{m}_{\eps}(\sigma) \coloneqq P^{m}_{\eps}(\tilde{t}_{\eps}(\sigma)) \,.
\end{eqnarray*}

By~\eqref{II-81} we have $\sigma=\tilde{t}_{\varepsilon}(\sigma)+|\tilde{l}_{\varepsilon}(\sigma)|_{1}-|l_{0}|_{1}$. Differentiating this relation we get
\begin{equation}\label{II-82}
\tilde{t}_{\varepsilon}'(\sigma)+|\tilde{l}_{\varepsilon}'(\sigma)|_{1}=1
\end{equation}
for every~$\varepsilon>0$ and a.e.~$\sigma\in[0, \sigma_{\eps}(T)]$. By~\eqref{II-82} and the monotonicity of~$\tilde{l}_{\varepsilon}^{m}$ we have $0\leq (\tilde{l}_{\varepsilon}^{m})'(\sigma) \leq 1 $ for every $\varepsilon>0$, every $m=1,\ldots,M$, and a.e.~$\sigma\in[0, \sigma_{\eps}(T)]$. Moreover,~$\tilde{t}_{\varepsilon}$ and~$\tilde{l}_{\varepsilon}$ are Lipschitz continuous.

We define $\widetilde{\G}_{\varepsilon}^{m}(\sigma)\coloneqq \G^{m}(\tilde{t}_{\varepsilon}(\sigma);\widetilde{\Gamma}_{\varepsilon}(\sigma))$ for $\sigma\in[0, \sigma_{\eps}(T)]$ and $ \bar S \coloneqq \sup_{\varepsilon>0}\sigma_{\eps}(T)$, which is bounded by a constant depending on~$T$ and on the class~$\adcr$. In order to deal with functions defined on the same interval, we extend $\tilde{t}_{\varepsilon}$, $\tilde{l}_{\varepsilon}$, $\widetilde\Gamma_{\eps}$, $\widetilde{\Gamma}^{m}_{\eps}$, $\tilde{t}_{\varepsilon}'$, and $\tilde{l}_{\varepsilon}'$ on $(\sigma_{\eps}(T) , \bar S]$ by $\tilde{t}_{\varepsilon}(\sigma)\coloneqq \tilde{t}_{\varepsilon}(\sigma_{\eps}(T))$, $\tilde{l}_{\varepsilon}(\sigma)\coloneqq \tilde{l}_{\varepsilon}(\sigma_{\eps}(T))$, $\widetilde{\Gamma}_{\eps}(\sigma)\coloneqq \widetilde{\Gamma}_{\eps}(\sigma_{\eps}(T))$, $\widetilde{\Gamma}^{m}_{\eps}(\sigma) \coloneqq \widetilde{\Gamma}^{m}_{\eps}(\sigma_{\eps}(T))$, $\tilde{t}_{\varepsilon}'(\sigma)\coloneqq 0$, and~$\tilde{s}_{\varepsilon}'(\sigma)\coloneqq 0$. 

Recalling that $\tilde{t}_{\varepsilon}'(\sigma)>0$ on $[0,\sigma_{\eps}(T)]$, the Griffith's criterion stated in Proposition~\ref{prop:viscous} reads in the new variables as
\begin{equation}\label{II-85}
\left\{ \begin{array}{lll}
 (\tilde{l}_{\varepsilon}^{m})' (\sigma ) \geq 0 \,, \\[1mm]
 \kappa(\widetilde{P}^{m}_{\eps}(\sigma)) \tilde{t}_{\varepsilon}'(\sigma) - \widetilde{\G}^{m}_{ \varepsilon}(\sigma) \, \tilde{t}_{\varepsilon}' (\sigma) + \varepsilon (\tilde{l}^{m}_{\varepsilon})' (\sigma) \geq 0\,, \\[1mm]
 \,(\tilde{l}^{m}_{\varepsilon})'(\sigma) \big( \kappa ( \widetilde{P}^{m}_{\eps} (\sigma) ) \tilde{t}_{\varepsilon}'(\sigma) - \widetilde{\G}^{m}_{ \varepsilon} (\sigma) \, \tilde{t}_{\varepsilon}' (\sigma) + \varepsilon (\tilde{l}^{m}_{\varepsilon})' (\sigma)\big )= 0\,,
 \end{array}\right.
\end{equation} 
for every $m$, every $\varepsilon$, and a.e.~$\sigma\in[0,\bar S]$.

Finally, we observe that by~(f) of Proposition~\ref{prop:conv-n}
\[
\begin{split}
\varepsilon \int_{0}^{ \sigma_{\eps} ( T ) }\!\! | (\tilde{l}^{m}_{\varepsilon})'(\sigma)|_{2}^{2}\,\di \sigma& = \varepsilon \int_{0}^{\sigma_{\eps}(T)} \! |\dot{l}^{m}_{\varepsilon}\,(\tilde{t}_{\varepsilon}(\sigma))|_{2}^{2}(\tilde{t}'_{\varepsilon})^{2}(\sigma)\,\di \sigma\\
&\displaystyle\leq\varepsilon \int_{0}^{\sigma_{\eps}(T)}\!|\dot{l}^{m}_{\varepsilon}(\tilde{t}_{\varepsilon}(\sigma))|_{2}^{2}\,\tilde{t}'_{\varepsilon}(\sigma)\,\di \sigma
=\varepsilon \int_{0}^{T}\!|\dot{l}^{m}_{\varepsilon}(t)|_{2}^{2}\,\di t\leq C
\end{split}
\]
uniformly in~$\eps$ and~$m\in\{1,\ldots,M\}$. Therefore, $\varepsilon (\tilde{l}^{m}_{\varepsilon})' \to 0$ in~$L^{2}( 0, \bar S )$.

Passing to the limit as $\eps\to 0$, we are now able to prove Theorem~\ref{thm:param}, showing that the parametrized solution~$\widetilde{\Gamma}$ satisfies a generalized Griffith's criterion.

\begin{proof}[Proof of Theorem~\ref{thm:param}]
%
%
%

Since $\widetilde{\Gamma}_{\eps}\colon[0, \bar S]\to \adcr$ is a nondecreasing set function with uniformly bounded length, there exists~$\widetilde{\Gamma}\colon [0,\bar S]\to \adcr$ such that, up to a not relabeled subsequence,~$\widetilde{\Gamma}_{\eps}(\sigma) \to \widetilde{\Gamma}(\sigma)$ and $\widetilde{\Gamma}^{m}_{\eps}(\sigma) \to \widetilde{\Gamma}^{m}(\sigma)$ in the Hausdorff metric of sets for every~$\sigma\in[0, \bar S]$ and every~$m\in\{1,\ldots,M\}$. We denote with~$\widetilde{P}^{m}(\sigma)$ the tip of~$\widetilde{\Gamma}^{m}(\sigma)$ and we notice that~$\widetilde{P}^{m}_{\eps}(\sigma) \to \widetilde{P}^{m}(\sigma)$ for every~$\sigma\in[0,\bar S]$ and every~$m\in\{1,\ldots,M\}$. 

Being $\tilde{t}_{\varepsilon},\tilde{l}^{m}_{\varepsilon}$ bounded in~$W^{1,\infty}(0, \bar S)$, up to a further subsequence we have that~$\tilde{t}_{\varepsilon}$ and~$\tilde{l}^{m}_{\varepsilon}$ converge weakly* in~$W^{1,\infty}( 0, \bar S )$ to some functions~$\tilde{t}$ and~$\tilde{l}^{m}$, respectively, As a consequence, we have that~$\tilde{l}^{m}(\sigma) = \Haus(\widetilde{\Gamma}^{m}(\sigma))$, so that~$\widetilde{\Gamma}^{m}\colon [0, S]\to \adcr$ is continuous in the Hausdorff metric of sets. We can also assume that $\sigma_{\eps}(T) \to S$ and $\tilde{t},\tilde{l}^{m}\in W^{1,\infty}( 0, S )$. Moreover, writing~\eqref{II-82} in an integral form and passing to the limit, we deduce that for a.e.~$\sigma\in[0, S]$
\begin{equation}\label{II-84}
\tilde{t}'(\sigma)+|\tilde{l}'(\sigma)|_{1}=1\,,
\end{equation}
where we have set~$\tilde{l}(\sigma)\coloneqq (l^{1}(\sigma),\ldots, \tilde{l}^{M}(\sigma))$. For $m=1\ldots,M$ and $\sigma\in[0, S]$ we define,
\begin{displaymath}
\widetilde{\G}^{m}(\sigma) \coloneqq \G^{m}(\tilde{t}(\sigma); \tilde\Gamma(\sigma) )\,,\qquad \widetilde{\G}(\sigma) \coloneqq (\widetilde{\G}^{1}(\sigma) , \ldots, \widetilde{\G}^{M}(\sigma) ) \,.
\end{displaymath}
We notice that, by Remark~\ref{r.ERR2}, $\widetilde{\G}_{\eps}(\sigma)$ converges to~$\widetilde{\G}(\sigma)$ for every~$\sigma\in[0, S]$ and~$\widetilde{\G}_{\eps} \to \widetilde{\G}$ in~$L^{2}( 0, S)$, as $\eps \to 0$.

By the monotonicity of~$\tilde{t}$ and~$\tilde{l}^{m}$, we have~$\tilde{t}'(\sigma)\geq0$ and~$(\tilde{l}^{m})'(\sigma)\geq0$ for every~$m$ and a.e.~$\sigma\in[0,  S]$. Moreover, by~\eqref{II-84} they can not be simultaneously zero.

%
Let us fix $m\in\{1,\ldots,M\}$ and $\psi\in L^{2}( 0, S )$ with $\psi\geq0$. Thanks to~\eqref{II-85}, for every~$\eps$ we have
\begin{equation}\label{II-86}
\int_{0}^{S }( \kappa(\widetilde{P}^{m}_{\eps}(\sigma)) \tilde{t}_{\varepsilon}'(\sigma) - \widetilde{\G}^{m}_{\varepsilon}(\sigma)\tilde{t}_{\varepsilon}'(\sigma) + \varepsilon (\tilde{l}^{m}_{\varepsilon})' ( \sigma ) ) \, \psi ( \sigma )\,\di \sigma \geq 0 \,.
\end{equation}

Since~$\tilde{t}'_{\varepsilon}$ converges to~$\tilde{t}'$ weakly* in $L^{\infty}( 0, S )$, $\varepsilon (\tilde{l}^{m}_{\varepsilon})' \to 0$ in $L^{2}( 0, S)$, $\widetilde{P}^{m}_{\eps}(\sigma) \to \widetilde{P}^{m}(\sigma)$ for $\sigma\in [0,S]$, and~$\widetilde{\G}^{m}_{\eps} \to \widetilde{\G}^{m}$ in~$L^{2}(0,S)$, passing to the limit in~\eqref{II-86} as $\eps \to 0$ we get
\[
   \int_{0}^{S} (  \kappa(\widetilde{P}^{m} (\sigma)) \tilde{t} '(\sigma)-\widetilde{\G}^{m}  (\sigma) \tilde{t} ' (\sigma) ) \, \psi (\sigma) \, \di \sigma \geq 0\,,
\]
which implies~(pG2).

We notice that if (pG1), (pG2) and~\eqref{II-84} hold, then (pG3) and (pG4) are equivalent to the following property:
\[
\text{if $\widetilde{\G}^{m}(\bar{\sigma}) < \kappa(\widetilde P^{m}(\bar{\sigma}))$ for some~$m$ and some~$\bar{\sigma}\in (0, S)$, then~$\tilde{l}^{m}$ is locally constant around~$\bar{\sigma}$.}
\]
Let us therefore assume that~$\widetilde{\G}^{m}(\bar{\sigma})< \kappa(\widetilde P^{m}(\bar{\sigma}))$. We first claim that there exist $\bar{\eps}>0$ and~$\delta>0$ such that $\widetilde{\G}^{m}_{\varepsilon}(\sigma)< \kappa(\widetilde{P}^{m}_{\eps}(\sigma))$ for every $\sigma \in (\bar{\sigma}-\delta,\bar{\sigma}+\delta)$ and  every $\eps \leq \bar{\eps}$. By contradiction, suppose that this is not the case. Then, there exist~$\sigma_{k}\to \bar{\sigma}$ and~$\eps_{k}\to 0$ such that~$\widetilde\G^{m}_{\eps_{k}}(\sigma_{k}) \geq \kappa(\widetilde{P}^{m}_{\eps_{k}}(\sigma_{k}))$. By continuity and monotonicity of~$\widetilde{\Gamma}^{m}$, we have that~$\widetilde{\Gamma}^{m}_{\eps_{k}}(\sigma_{k}) \to \widetilde{\Gamma}^{m}(\sigma)$ in the Hausdorff metric of sets and~$\widetilde{P}^{m}_{\eps_{k}}(\sigma_{k}) \to \widetilde{P}^{m}(\sigma)$ as~$k\to \infty$. Hence, the continuity of the energy release rate and the hypothesis~(H4) lead us to the contradiction $\widetilde\G^{m}( \bar \sigma)\geq \kappa(\widetilde{P}^{m}(\bar{\sigma}))$.

Let~$\delta$ and~$\bar{\eps}$ be as above. We deduce from the Griffith's criterion~\eqref{II-85} that~$\tilde{l}^{m}_{\varepsilon}$ is constant in $(\bar{\sigma}-\delta,\bar{\sigma}+\delta)$ for every~$\eps\leq \bar{\eps}$. Since~$\tilde{l}^{m}_{\varepsilon}$ converges to~$\tilde{l}^{m}$ weakly* in $W^{1,\infty}( 0, S )$, we get that~$\tilde{l}^{m}$ is locally constant around~$\bar{\sigma}$, and this concludes the proof of~(pG3) and~(pG4).

In order to show that~$\Gamma(\tilde{t}(\sigma)) = \widetilde{\Gamma}(\sigma)$ for every~$\sigma\in[0,S]$ such that~$\tilde{t}'(\sigma)>0$, we define
\begin{displaymath}
s(t)\coloneqq \min\,\{s\in[0,S]:\, \tilde{t}(s) = t\}\qquad\text{for every~$t\in[0,T]$}\,.
\end{displaymath}
If $\tilde{t}'(\sigma)>0$, then we have~$s(\tilde{t}(\sigma)) = \sigma$ and~$s(\tilde{t}(\bar{\sigma})) \neq s(\tilde{t}(\sigma))$ for $\bar{\sigma}\neq \sigma$. Let us prove that~$\tilde{t}(\sigma)$ is a continuity point for~$\Gamma$, where the map~$ t \mapsto \Gamma(t)$ has been determined in Theorem~\ref{thm:bv}. By contradiction, assume that~$\tilde{t}(\sigma)$ is a discontinuity point of~$\Gamma$. Then, there exist $t_{\eps}^{1} < t_{\eps}^{2}$ such that $t_{\eps}^{1},t_{\eps}^{2}\to \tilde{t}(\sigma)$ and $\Gamma_{\eps}(t^{1}_{\eps}) \to \Gamma^{-}(\tilde{t}(\sigma))$ and~$\Gamma_{\eps}(t^{2}_{\eps}) \to \Gamma^{+}(\tilde{t}(\sigma))$ in the Hausdorff metric of sets, where we have denoted with~$\Gamma^{\pm}(\tilde{t}(\sigma))$ the left and right limits of~$\Gamma(t)$ in~$\tilde{t}(\sigma)$. As a consequence, $s(\tilde{t}(\bar{\sigma})) = s(\tilde{t}(\sigma))$ for~$\bar{\sigma} \in (\sigma - \Haus(\Gamma^{+}(\tilde{t}(\sigma))\setminus\Gamma^{-}(\tilde{t}(\sigma))), \sigma + \Haus(\Gamma^{+}(\tilde{t}(\sigma))\setminus\Gamma^{-}(\tilde{t}(\sigma))))$, which is a contradiction. Hence,~$\tilde{t}(\sigma)$ is a continuity point of~$t\mapsto \Gamma(t)$. Therefore,~$\widetilde{\Gamma}_{\eps}(\sigma) = \Gamma_{\eps} (\tilde{t}_{\eps}(\sigma))$ converges to~$\Gamma(\tilde{t}(\sigma))$ in the Hausdorff metric of sets. This implies that~$\Gamma(\tilde{t}(\sigma)) = \widetilde{\Gamma}(\sigma)$.

We conclude with the energy balance~\eqref{e:vbalance2}. Let $s\in[0,S]$. By the change of variable~$t= t_{\eps}(\sigma)$ in~\eqref{e:vbalance}, for~$\eps>0$ we have
\begin{equation}\label{e:vbalance3}
\begin{split}
\F (\tilde{t}_{\eps}(s) ; \widetilde{\Gamma}_{\eps}(s))  = & \ \F(0;\Gamma_{0}) 
+ \int_{0}^{s}\int_{\Om} \C\e \tilde{u}_{\eps}(\sigma) : \e\dot{w}(\tilde{t}_{\eps}(\sigma))\,\tilde{t}'_{\eps} (\sigma)\,\di x \, \di \sigma \\
& -  \sum_{m=1}^M \int_{0}^{s} (\widetilde{\G}^{m}_{\eps}(\sigma) - \kappa( \widetilde{P}^{m}_{\eps}(\sigma))) (\tilde{l}_{\eps}^{m})'(\sigma)\,\di \sigma \\
& - \int_{0}^{s}\int_{\Om} \dot{f}(\tilde{t}_{\eps}(\sigma)) \cdot \tilde{u}_{\eps}(\sigma)\,\tilde{t}'_{\eps} (\sigma)\,\di x \,\di \sigma 
- \int_{0}^{s} \int_{\Om} f(\tilde{t}_{\eps}(\sigma))\cdot \dot{w}(\tau) \,\tilde{t}'_{\eps} (\sigma)\, \di x  \, \di \sigma \\
& - \int_{0}^{s} \int_{\partial_{S}\Om} \dot g(\tilde{t}_{\eps}(\sigma)) \cdot \tilde{u}(\sigma)\,\tilde{t}'_{\eps} (\sigma) \, \di \Haus  \, \di \sigma 
- \int_{0}^{s}\int_{\partial_{S}\Om} g(\tilde{t}_{\eps}(\sigma) )\cdot \dot{w} (\tilde{t}_{\eps}(\sigma))\,\tilde{t}'_{\eps} (\sigma)\,\di \Haus \, \di\sigma\,,
\end{split}
\end{equation}
where we have set $\tilde{u}_{\eps}(\sigma)\coloneqq u_{\eps}(\tilde{t}_{\eps}(\sigma))$. Since~$\tilde{t}_{\eps}$ and~$\tilde{l}^{m}_{\eps}$ converge weakly* in~$W^{1,\infty}(0,S)$ to~$\tilde{t}$ and~$\tilde{l}^{m}$  and~$\widetilde{\G}^{m}_{\eps} \to \widetilde{\G}^{m}$ in~$L^{2}(0,S)$, passing to the limit as~$\eps\to 0$ in~\eqref{e:vbalance3} we get~\eqref{e:vbalance2}. This concludes the proof of the theorem.
\end{proof}

%
%
%
%
%

\subsection*{Acknowledgments}
The authors would like to acknowledge the kind hospitality of the Erwin Schr\"odinger International Institute for Mathematics and Physics (ESI), where part of this research was developed during the workshop {\it New trends in the variational modeling of failure phenomena}.
All authors would like to acknowledge the kind hospitality of the University of Naples Federico II, to which GL was affiliated when this research was initiated. 
SA wishes to thank the Technical University of Munich, where he worked during the preparation of this paper,
with partial support from the SFB project TRR109 \emph{Shearlet approximation of brittle fracture evolutions}. 
GL and IL are members of the {\em Gruppo Nazionale per l'Analisi Ma\-te\-ma\-ti\-ca, la Probabilit\`a e le loro Applicazioni} (GNAMPA) of the {\em Istituto Nazionale di Alta Matematica} (INdAM). 
GL received support from the 
INdAM-GNAMPA 2018 Project \emph{Analisi va\-ria\-zio\-na\-le per difetti e interfacce nei materiali}
and through the 2019 Project \emph{Modellazione e studio di propriet\`a asintotiche per problemi variazionali in fenomeni anelastici}.

\bibliographystyle{siam}
\bibliography{A_La_Lu_18bib}

\end{document}